\documentclass[a4paper]{amsart}

\usepackage[utf8]{inputenc}
\usepackage[T1]{fontenc}
\usepackage[english]{babel}

\usepackage{amsfonts}
\usepackage{amsmath}
\usepackage{amsrefs}
\usepackage{amssymb}
\usepackage{amstext}
\usepackage{amsthm}

\usepackage{geometry}
\usepackage{fullpage}

\usepackage[shortlabels]{enumitem}
\usepackage{graphicx}
\usepackage{hyperref}
\usepackage{listings}
\usepackage{mathrsfs}
\usepackage{mathtools}
\usepackage{multicol}
\usepackage{shuffle}

\usepackage{tikz}
\usepackage{tikz-cd}
\usetikzlibrary{arrows,backgrounds,chains,decorations.pathreplacing,patterns,shapes}

\newcommand{\area}{\mathsf{area}}
\newcommand{\dinv}{\mathsf{dinv}}

\newcommand{\Val}{\mathsf{Val}}
\newcommand{\Rise}{\mathsf{Rise}}
\newcommand{\Peak}{\mathsf{Peak}}
\newcommand{\ZVal}{\mathsf{ZVal}}
\newcommand{\DRise}{\mathsf{DRise}}
\newcommand{\DPeak}{\mathsf{DPeak}}

\newcommand{\D}{\mathsf{D}} 
\newcommand{\SQE}{\mathsf{SQ^{E}}} 
\newcommand{\PLSQE}{\mathsf{PLSQ^{E}}} 
 
\newcommand{\PLD}{\mathsf{PLD}} 

\newcommand{\gendelta}{\mathsf{PLD}_{\underline x,q,t}(m,n)^{\ast k}} 
\newcommand{\gensquare}{\mathsf{PLSQ^E}_{\underline x,q,t}(m,n)^{\ast k}} 

\newcommand{\qbinom}[2]{\genfrac{[}{]}{0pt}{}{#1}{#2}}

\makeatletter

\pgfkeys{
	/tikz/sharp angle/.code={%
		\pgfsetarrowoptions{sharp >}{#1}%
		\pgfsetarrowoptions{sharp <}{-#1}%
	},
	/tikz/sharp > angle/.code={%
		\pgfsetarrowoptions{sharp >}{#1}%
	},
	/tikz/sharp < angle/.code={%
		\pgfsetarrowoptions{sharp <}{#1}%
	},
	/tikz/sharp protrude/.code=\csname if#1\endcsname\qrr@tikz@sharp@z@-0.05\p@\else\qrr@tikz@sharp@z@\z@\fi,
	/tikz/sharp protrude/.default=true
}

\newdimen\qrr@tikz@sharp@z@
\qrr@tikz@sharp@z@\z@
\pgfarrowsdeclare{sharp >}{sharp >}{%
	\edef\pgf@marshal{\noexpand\pgfutil@in@{and}{\pgfgetarrowoptions{sharp >}}}%
	\pgf@marshal
	\ifpgfutil@in@
	\edef\pgf@tempa{\pgfgetarrowoptions{sharp >}}
	\expandafter\qrr@tikz@sharp@parse\pgf@tempa\@qrr@tikz@sharp@parse
	\else
	\qrr@tikz@sharp@parse\pgfgetarrowoptions{sharp >}and-\pgfgetarrowoptions{sharp >}\@qrr@tikz@sharp@parse
	\fi
	\pgfmathparse{max(\pgf@tempa,\pgf@tempb,0)}%
	\let\qrr@tikz@sharp@max\pgfmathresult
	\pgfmathsetlength\pgf@xa{.5*\pgflinewidth * tan(\qrr@tikz@sharp@max)}%
	\pgfarrowsleftextend{+\pgf@xa}%
	\pgfarrowsrightextend{+\pgf@xa}%
}{%
	\edef\pgf@marshal{\noexpand\pgfutil@in@{and}{\pgfgetarrowoptions{sharp >}}}%
	\pgf@marshal
	\ifpgfutil@in@
	\edef\pgf@tempa{\pgfgetarrowoptions{sharp >}}
	\expandafter\qrr@tikz@sharp@parse\pgf@tempa\@qrr@tikz@sharp@parse
	\else
	\qrr@tikz@sharp@parse\pgfgetarrowoptions{sharp >}and-\pgfgetarrowoptions{sharp >}\@qrr@tikz@sharp@parse
	\fi
	\pgfmathsetlength\pgf@ya{.5*\pgflinewidth * tan(max(\pgf@tempa,\pgf@tempb,0))}%
	\pgfmathsetlength\pgf@xa{-.5*\pgflinewidth * tan(\pgf@tempa)}%
	\pgfmathsetlength\pgf@xb{-.5*\pgflinewidth * tan(\pgf@tempb)}%
	\advance\pgf@xa\pgf@ya
	\advance\pgf@xb\pgf@ya
	\ifdim\pgf@xa>\pgf@xb
	\pgftransformyscale{-1}%
	\pgf@xc\pgf@xb
	\pgf@xb\pgf@xa
	\pgf@xa\pgf@xc
	\fi
	\pgfpathmoveto{\pgfqpoint{\qrr@tikz@sharp@z@}{.5\pgflinewidth}}%
	\pgfpathlineto{\pgfqpoint{\pgf@xa}{.5\pgflinewidth}}%
	\pgfpathlineto{\pgfqpoint{\pgf@ya}{+0pt}}%
	\pgfpathlineto{\pgfqpoint{\pgf@xb}{-.5\pgflinewidth}}%
	\pgfpathlineto{\pgfqpoint{\qrr@tikz@sharp@z@}{-.5\pgflinewidth}}%
	\pgfusepathqfill
}
\pgfarrowsdeclare{sharp <}{sharp <}{%
	\edef\pgf@marshal{\noexpand\pgfutil@in@{and}{\pgfgetarrowoptions{sharp <}}}%
	\pgf@marshal
	\ifpgfutil@in@
	\edef\pgf@tempa{\pgfgetarrowoptions{sharp <}}
	\expandafter\qrr@tikz@sharp@parse\pgf@tempa\@qrr@tikz@sharp@parse
	\else
	\expandafter\qrr@tikz@sharp@parse\pgfgetarrowoptions{sharp <}and-\pgfgetarrowoptions{sharp <}\@qrr@tikz@sharp@parse
	\fi
	\pgfmathparse{max(\pgf@tempa,\pgf@tempb,0)}%
	\let\qrr@tikz@sharp@max\pgfmathresult
	\pgfmathsetlength\pgf@xa{.5*\pgflinewidth * tan(\qrr@tikz@sharp@max)}%
	\pgfarrowsleftextend{+\pgf@xa}%
	\pgfarrowsrightextend{+\pgf@xa}%
}{%
	\edef\pgf@marshal{\noexpand\pgfutil@in@{and}{\pgfgetarrowoptions{sharp <}}}%
	\pgf@marshal
	\ifpgfutil@in@
	\edef\pgf@tempa{\pgfgetarrowoptions{sharp <}}
	\expandafter\qrr@tikz@sharp@parse\pgf@tempa\@qrr@tikz@sharp@parse
	\else
	\expandafter\qrr@tikz@sharp@parse\pgfgetarrowoptions{sharp <}and-\pgfgetarrowoptions{sharp <}\@qrr@tikz@sharp@parse
	\fi
	\pgfmathsetlength\pgf@ya{.5*\pgflinewidth * tan(max(\pgf@tempa,\pgf@tempb,0))}%
	\pgfmathsetlength\pgf@xa{-.5*\pgflinewidth * tan(\pgf@tempa)}%
	\pgfmathsetlength\pgf@xb{-.5*\pgflinewidth * tan(\pgf@tempb)}%
	\advance\pgf@xa\pgf@ya
	\advance\pgf@xb\pgf@ya
	\ifdim\pgf@xa>\pgf@xb
	\pgftransformyscale{-1}%
	\pgf@xc\pgf@xb
	\pgf@xb\pgf@xa
	\pgf@xa\pgf@xc
	\fi
	\pgfpathmoveto{\pgfqpoint{\qrr@tikz@sharp@z@}{.5\pgflinewidth}}%
	\pgfpathlineto{\pgfqpoint{\pgf@xa}{.5\pgflinewidth}}%
	\pgfpathlineto{\pgfqpoint{\pgf@ya}{+0pt}}%
	\pgfpathlineto{\pgfqpoint{\pgf@xb}{-.5\pgflinewidth}}%
	\pgfpathlineto{\pgfqpoint{\qrr@tikz@sharp@z@}{-.5\pgflinewidth}}%
	\pgfusepathqfill
}
\def\qrr@tikz@sharp@parse#1and#2\@qrr@tikz@sharp@parse{\def\pgf@tempa{#1}\def\pgf@tempb{#2}}

\makeatother

\newcommand\multiset[2]%
{\mathchoice{\left(\kern-0.4em{\binom{#1}{#2}}\kern-0.4em\right)}
	{\bigl(\kern-0.2em{\binom{#1}{#2}}\kern-0.2em\bigr)}
	{\bigl(\kern-0.2em{\binom{#1}{#2}}\kern-0.2em\bigr)}
	{\bigl(\kern-0.2em{\binom{#1}{#2}}\kern-0.2em\bigr)}}

\let\existstemp\exists \renewcommand*{\exists}{\mathop \existstemp}
\let\foralltemp\forall \renewcommand*{\forall}{\mathop \foralltemp}

\def\quotient#1#2{\raise1ex\hbox{$#1$}\Big/\lower1ex\hbox{$#2$}}

\newcommand{\<}{\langle}
\renewcommand{\>}{\rangle}

\newtheorem{theorem}{Theorem}[section]
\newtheorem{lemma}[theorem]{Lemma}
\newtheorem{proposition}[theorem]{Proposition}
\newtheorem{corollary}[theorem]{Corollary}
\newtheorem{conjecture}[theorem]{Conjecture}

\theoremstyle{definition}
\newtheorem{definition}[theorem]{Definition}

\newtheorem{example}[theorem]{Example}

\theoremstyle{remark}
\newtheorem{remark}[theorem]{Remark}

\title{The Delta square conjecture}

\author{Michele D'Adderio}
\address{Universit\'e Libre de Bruxelles (ULB)\\D\'epartement de Math\'ematique\\ Boulevard du Triomphe, B-1050 Bruxelles\\ Belgium}\email{mdadderi@ulb.ac.be}

\author{Alessandro Iraci}
\address{Universit\'a di Pisa and Universit\'e Libre de Bruxelles (ULB)\\Dipartimento di Matematica\\ Largo Bruno Pontecorvo 5, 56127 Pisa\\ Italia}\email{iraci@student.dm.unipi.it}

\author{Anna Vanden Wyngaerd}
\address{Universit\'e Libre de Bruxelles (ULB)\\D\'epartement de Math\'ematique\\ Boulevard du Triomphe, B-1050 Bruxelles\\ Belgium}\email{anvdwyng@ulb.ac.be}

\begin{document}
	
\begin{abstract}
	We conjecture a formula for the symmetric function $\frac{[n-k]_t}{[n]_t}\Delta_{h_m}\Delta_{e_{n-k}}\omega(p_n)$ in terms of decorated partially labelled square paths. This can be seen as a generalization of the square conjecture of Loehr and Warrington \cite{Loehr-Warrington-square-2007}, recently proved by Sergel \cite{Leven-2016} after the breakthrough of Carlsson and Mellit \cite{Carlsson-Mellit-ShuffleConj-2015}. Moreover, it extends to the square case the combinatorics of the generalized Delta conjecture of Haglund, Remmel and Wilson \cite{Haglund-Remmel-Wilson-2015}, answering one of their questions. We support our conjecture by proving the specialization $m=q=0$, reducing it to the same case of the Delta conjecture, and the Schr\"oder case, i.e. the case $\<\cdot ,e_{n-d}h_d\>$. The latter provides a broad generalization of the $q,t$-square theorem of Can and Loehr \cite{Can-Loehr-2006}. We give also a combinatorial involution, which allows to establish a linear relation among our conjectures (as well as the generalized Delta conjectures) with fixed $m$ and $n$. Finally, in the appendix, we give a new proof of the Delta conjecture at $q=0$.
\end{abstract}
	
\maketitle
\tableofcontents

\section{Introduction}

In \cite{Haglund-Remmel-Wilson-2015}, Haglund, Remmel and Wilson conjectured a combinatorial formula for $\Delta_{e_{n-k-1}}'e_n$ in terms of decorated labelled Dyck paths, which they called \emph{Delta conjecture}, after the so called delta operators $\Delta_f'$ introduced by Bergeron, Garsia, Haiman, and Tesler \cite{Bergeron-Garsia-Haiman-Tesler-Positivity-1999} for any symmetric function $f$. In fact in the same article \cite{Haglund-Remmel-Wilson-2015} the authors conjectured a combinatorial formula for the more general $\Delta_{h_m}\Delta_{e_{n-k-1}}'e_n$ in terms of decorated partially labelled Dyck paths, which we call \emph{generalized Delta conjecture}.

These problems have attracted considerable attention since their formulation: a partial list of works about the Delta conjecture is \cite{Haglund-Remmel-Wilson-2015,DAdderio-VandenWyngaerd-2017,DAdderio-Iraci-polyominoes-2017,Haglund-Rhoades-Shimozono-Advances,Garsia-Haglund-Remmel-Yoo-2017,Remmel-Wilson-2015,Wilson-Equidistribution,Rhoades-2018,Romero-Deltaq1-2017,Zabrocki-4Catalan-2016}. The main result about the generalized Delta conjecture is the proof of the Schr\"{o}der case, i.e. the case $\<\cdot , e_{n-d}h_d\>$, in \cite{DAdderio-Iraci-VandenWyngaerd-GenDeltaSchroeder}.

The special case $k=0$ of the Delta conjecture, that has been known as the \emph{Shuffle conjecture} \cite{HHLRU-2005}, was recently proved by Carlsson and Mellit \cite{Carlsson-Mellit-ShuffleConj-2015}. The latter turns out to be a combinatorial formula for the Frobenius characteristic of the $\mathfrak{S}_n$-module of diagonal harmonics studied by Garsia and Haiman in relation to the famous \emph{$n!$ conjecture}, now $n!$ theorem of Haiman \cite{Haiman-nfactorial-2001}.

\medskip

In \cite{Loehr-Warrington-square-2007} Loehr and Warrington conjectured a combinatorial formula for $\Delta_{e_{n}}\omega(p_n)=\nabla \omega(p_n)$ in terms of labelled square paths (ending east), called \emph{square conjecture}. The special case $\<\cdot ,e_n\>$ of this conjecture, known as \emph{$q,t$-square}, has been proved earlier by Can and Loehr in \cite{Can-Loehr-2006}. Recently the full square conjecture has been proved by Sergel in \cite{Leven-2016} after the breakthrough of Carlsson and Mellit in \cite{Carlsson-Mellit-ShuffleConj-2015}.

\medskip

In the present work we conjecture a combinatorial formula for $\frac{[n-k]_t}{[n]_t}\Delta_{h_m}\Delta_{e_{n-k}}\omega(p_n)$ in terms of \emph{decorated partially labelled square paths} that we call \emph{generalized Delta square conjecture}. In analogy with the Delta conjecture in \cite{Haglund-Remmel-Wilson-2015}, we call simply \emph{Delta square conjecture} the special case $m=0$. Our conjecture extends the square conjecture of Loehr and Warrington \cite{Loehr-Warrington-square-2007} (now a theorem \cite{Leven-2016}), i.e. it reduces to that one for $m=k=0$. Moreover, it extends the generalized Delta conjecture in the sense that on decorated partially labelled Dyck paths gives the same combinatorial statistics. Notice that our conjecture answers a question in \cite{Haglund-Remmel-Wilson-2015}.

\medskip

In the present work we support our conjecture by proving some of its consequences. In particular, we prove the Delta square conjecture (i.e. the case $m=0$) at $q=0$: this turns out to reduce to the specialization $q=0$ of the Delta conjecture, already proved in \cite{Garsia-Haglund-Remmel-Yoo-2017}. In fact, in the Appendix we provide a new proof of this result. Also, we prove the Schr\"{o}der case, i.e. the case $\<\cdot , e_{n-d}h_d\>$, of the generalized Delta square conjecture: this is a broad generalization of the $q,t$-square theorem of Can and Loehr \cite{Can-Loehr-2006}, and it is the analogue of the same result for the generalized Delta conjecture proved in \cite{DAdderio-Iraci-VandenWyngaerd-GenDeltaSchroeder}. Finally, we provide a combinatorial involution among the objects of the generalized Delta (square) conjectures for fixed $m$ and $n$. Together with its symmetric function counterpart and the specialization $q=0$ of the generalized Delta conjecture at $k=0$, this will prove a curious linear relation among such conjectures.

\medskip

The paper is organized as follows. In Section~2 we recall the generalized Delta conjecture of \cite{Haglund-Remmel-Wilson-2015} by giving the definitions and fixing the notation. In Section~3 we state our generalized Delta conjecture, and we make a few basic remarks. In Section~4 we fix the notation on symmetric functions and we prove the identities needed in the rest of the paper. In Section~5 we prove the Delta square conjecture (i.e. the case $m=0$) at $q=0$, by reducing it to the Delta conjecture at $q=0$. We will give a new proof of the latter in the Appendix: this in order to make our treatment more self-contained, but also because the new proof might have some independent interest. In Section~6 we prove the generalized Delta conjecture of \cite{Haglund-Remmel-Wilson-2015} at $k=0$ and $t=0$. In Section~7 we prove the Schr\"{o}der case, i.e. the case $\<\cdot, e_{n-d}h_d\>$ of our generalized Delta square conjecture. This is the analogue of the same result for the generalized Delta conjecture proved in \cite{DAdderio-Iraci-VandenWyngaerd-GenDeltaSchroeder}, and it is a broad generalization of the $q,t$-square theorem proved in \cite{Can-Loehr-2006}. In Section~8 we give a combinatorial involution that will provide a counterpart of two theorems on symmetric functions proved in Section~5. With this we will prove a curious linear relation among the Delta (square) conjectures for fixed $m$ and $n$. Finally in Section~9 we mention some open problems.

\section{The generalized Delta conjecture}

\emph{We refer to Section~\ref{sec:SF} for notations and definitions concerning symmetric functions.}

\medskip

In \cite{Haglund-Remmel-Wilson-2015}, the authors conjectured a combinatorial interpretation for the symmetric function \[ \Delta_{h_m}\Delta'_{e_{n-k-1}}e_n \] in terms of partially labelled decorated Dyck paths, known as the \emph{generalized Delta conjecture} because it reduces to the Delta conjecture when $m=0$. We give the necessary definitions.

\begin{definition}
	A \emph{Dyck path} of size $n$ is a lattice path going from $(0,0)$ to $(n,n)$, using only north and east unit steps and staying weakly above the line $x=y$ (also called the \emph{main diagonal}). The set of Dyck paths of size $n$ will be denoted by $\mathsf{D}(n)$. A \emph{partially labelled Dyck path} is a Dyck path whose vertical steps are labelled with (not necessarily distinct) non-negative integers such that the labels appearing in each column are strictly increasing from bottom to top, and $0$ does not appear in the first column. The set of partially labelled Dyck paths with $m$ zero labels and $n$ nonzero labels is denoted by $\PLD(m,n)$.  
\end{definition}

Partially labelled Dyck paths differ from labelled Dyck paths only in that $0$ is allowed as a label in the former and not in the latter. 

\begin{definition}\label{def: monomial path}
	We define for each $D\in \PLD(m,n)$ a monomial in the variables $x_1,x_2,\dots$: we set \[ x^D \coloneqq \prod_{i=1}^n x_{l_i(D)} \] where $l_i(D)$ is the label of the $i$-th vertical step of $D$ (the first being at the bottom). Notice that $x_0$ does not appear, which explains the word \emph{partially}.
\end{definition}

\begin{definition}
	Let $D$ be a (partially labelled) Dyck path of size $n+m$. We define its \emph{area word} to be the string of integers $a(D) = a_1(D) \cdots a_{n+m}(D)$ where $a_i(D)$ is the number of whole squares in the $i$-th row (counting from the bottom) between the path and the main diagonal.
\end{definition}

\begin{definition} \label{def: rise}
	The \emph{rises} of a Dyck path $D$ are the indices \[ \Rise(D) \coloneqq \{2\leq i \leq n+m\mid a_{i}(D)>a_{i-1}(D)\},\] or the vertical steps that are directly preceded by another vertical step. Taking a subset $\DRise(D)\subseteq \Rise (D)$ and decorating the corresponding vertical steps with a $\ast$, we obtain a \emph{decorated Dyck path}, and we will refer to these vertical steps as \emph{decorated rises}. 
\end{definition}

\begin{definition}
	Given a partially labelled Dyck path, we call \emph{zero valleys} its vertical steps with label $0$ (which are necessarily preceded by an horizontal step, that is why we call them valleys).
\end{definition}

The set of partially labelled decorated Dyck paths with $m$ zero labels, $n$ nonzero labels and $k$ decorated rises is denoted by $\PLD(m,n)^{\ast k}$. See Figure~\ref{fig:pldExample1} for an example. 

\begin{figure*}[!ht]
	\centering
	\begin{tikzpicture}[scale = .8]
	
	\draw[step=1.0, gray!60, thin] (0,0) grid (8,8);
	
	\draw[gray!60, thin] (0,0) -- (8,8);
	
	\draw[blue!60, line width=2pt] (0,0) -- (0,1) -- (0,2) -- (1,2) -- (2,2) -- (2,3) -- (2,4) -- (2,5) -- (3,5) -- (4,5) -- (4,6) -- (4,7) -- (4,8) -- (5,8) -- (6,8) -- (7,8) -- (8,8);
	
	\draw (0.5,0.5) circle (0.4 cm) node {$1$};
	\draw (0.5,1.5) circle (0.4 cm) node {$3$};
	\draw (2.5,2.5) circle (0.4 cm) node {$0$};
	\draw (2.5,3.5) circle (0.4 cm) node {$4$};
	\draw (2.5,4.5) circle (0.4 cm) node {$6$};
	\draw (4.5,5.5) circle (0.4 cm) node {$0$};
	\draw (4.5,6.5) circle (0.4 cm) node {$2$};
	\draw (4.5,7.5) circle (0.4 cm) node {$6$};
	
	\node at (1.5,3.5) {$\ast$};
	\node at (3.5,6.5) {$\ast$};
	
	\end{tikzpicture}
	\caption{Example of an element in $\PLD(2,6)^{\ast 2}$.}
	\label{fig:pldExample1}
\end{figure*}
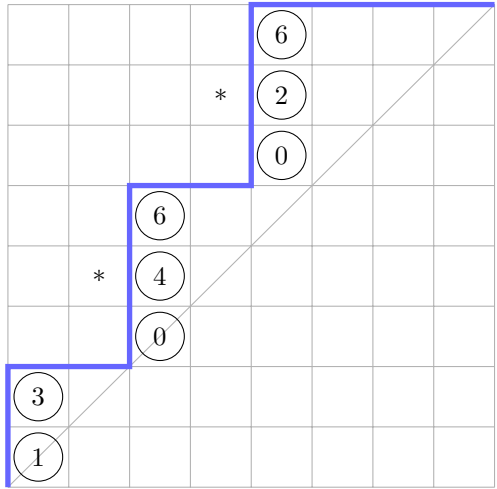

We define two statistics on this set.

\begin{definition} \label{def: area DP}
	We define the \emph{area} of a (partially labelled) decorated Dyck path $D$ as \[ \area(D) \coloneqq \sum_{i\not \in \DRise(D)} a_i(D). \]
\end{definition}

For a more visual definition, the area is the number of whole squares that lie between the path and the main diagonal, except for the ones in the rows containing a decorated rise. For example, the decorated Dyck path in Figure~\ref{fig:pldExample1} has area $7$. 

Notice that the area does not depend on the labels. 

\begin{definition} \label{def: dinv DP}
	Let $D \in \PLD(m,n)$. For $1 \leq i < j \leq n+m$, we say that the pair $(i,j)$ is an \emph{inversion} if
	\begin{itemize}
		\item either $a_i(D) = a_j(D)$ and $l_i(D) < l_j(D)$ (\emph{primary inversion}),
		\item or $a_i(D) = a_j(D) + 1$ and $l_i(D) > l_j(D)$ (\emph{secondary inversion}),
	\end{itemize}
	where $l_i(D)$ denotes the label of the vertical step in the $i$-th row.
	
	Then we define \[\dinv(D)\coloneqq \# \{ 0\leq i < j \leq n+m \mid (i,j) \; \text{is an inversion} \}.\]
\end{definition}

For example, the decorated Dyck path in Figure~\ref{fig:pldExample1} has $1$ primary inversion (the pair $(2,4)$) and $2$ secondary inversions (the pairs $(2,3)$ and $(5,6)$), so its dinv is $3$. 

Notice that the decorations on the rises do not affect the dinv.

\begin{definition}
	We define a formal series in the variables $\underline x=(x_1,x_2,\dots)$ and coefficients in $\mathbb N [q,t]$ 
	\[
	\gendelta \coloneqq \sum_{D\in \PLD(m,n)^{\ast k}}q^{\dinv(D)} t^{\area(D)}x^D.
	\]
\end{definition}

The following conjecture is stated in \cite{Haglund-Remmel-Wilson-2015}.
\begin{conjecture}[Generalized Delta]
	For $m,n,k\in \mathbb{N}$, $m\geq 0$ and $n>k\geq 0$,
	\[ \Delta_{h_{m}} \Delta'_{e_{n-k-1}} e_{n} = \gendelta . \]
\end{conjecture}
Notice that $\gendelta$ is in fact a symmetric function (cf. Remark~\ref{rmk:symmetry}).

\section{The generalized Delta square conjecture}

\emph{We refer to Section~\ref{sec:SF} for notations and definitions concerning symmetric functions.}

\medskip

\begin{definition}
	A \emph{square path ending east} of size $n$ is a lattice paths going from $(0,0)$ to $(n,n)$ consisting of east or north unit steps, always ending with an east step. The set of such paths is denoted by $\SQE(n)$. We call \emph{base diagonal} of a square path the diagonal $y=x+k$ with the smallest value of $k$ that is touched by the path (so that $k\leq 0$). The \emph{shift} of the square path is the non-negative value $-k$. The \emph{breaking point} of the square path is the lowest point in which the path touches the base diagonal (so for Dyck paths is $(0,0)$).
\end{definition}

For example, the path in Figure~\ref{fig: labelled square} has shift $3$.

\begin{definition}
	A \emph{partially labelled square path ending east} is a square path ending east whose vertical steps are labelled with (not necessarily distinct) non-negative integers such that the labels appearing in each column are strictly increasing bottom to top, there is at least one nonzero label labelling a vertical step starting from the base diagonal, and if the path starts with a vertical step, this first step's label is nonzero. The set of partially labelled square paths ending east with $m$ zero labels and $n$ nonzero labels is denoted by $\PLSQE(m,n)$. 
\end{definition}

\begin{definition}
	Let $P$ be a (partially labelled) square path ending east of size $n+m$. We define its \emph{area word} to be the string of integers $a(P) = a_1(P) \cdots a_{n+m}(P)$ where the $i$-th vertical step of the path starts from the diagonal $y=x+a_i(P)$. For example the path in Figure~\ref{fig: labelled square} has area word $0\, -3\, -3\, -2\, -2\, -1\, 0\, 0$.
\end{definition}

\begin{definition}
	Let $P$ be a partially labelled square path ending east. We define the monomial $x^P$ in the same way as for partially labelled Dyck paths (see Definition~\ref{def: monomial path}). 
\end{definition}

\begin{definition}
	The \emph{rises} of a square path ending east $P$ are defined in the same way as the rises of a Dyck path (see Definition~\ref{def: rise}). Taking a subset $\DRise(P)\subseteq \Rise (P)$ and decorating the corresponding vertical steps with a $\ast$, we obtain a \emph{decorated  square path}, and we will refer to these vertical steps as \emph{decorated rises}.
\end{definition}

\begin{definition}
	Given a partially labelled square path, we call \emph{zero valleys} its vertical steps with label $0$ (which are necessarily preceded by a horizontal step, hence the name valleys).
\end{definition}

The set of partially labelled decorated square pahts ending east with $m$ zero labels, $n$ nonzero labels and $k$ decorated rises is denoted by $\PLSQE(m,n)^{\ast k}$. See Figure~\ref{fig: labelled square} for an example.

\begin{figure}[ht]
	\begin{tikzpicture}[scale = 0.8]
	\draw[step=1.0, gray!60, thin] (0,0) grid (8,8);
	
	\draw[gray!60, thin] (3,0) -- (8,5);
	
	\draw[blue!60, line width=1.6pt] (0,0) -- (0,1) -- (1,1) -- (2,1) -- (3,1) -- (4,1) -- (4,2) -- (5,2) -- (5,3) -- (5,4) -- (6,4) -- (6,5) -- (6,6) -- (6,7) -- (7,7) -- (7,8) -- (8,8);
	
	\node at (5.5,5.5) {$\ast$};
	
	\node at (0.5,0.5) {$2$};
	\draw (0.5,0.5) circle (.4cm); 
	\node at (4.5,1.5) {$0$};
	\draw (4.5,1.5) circle (.4cm); 
	\node at (5.5,2.5) {$2$};
	\draw (5.5,2.5) circle (.4cm); 
	\node at (5.5,3.5) {$4$};
	\draw (5.5,3.5) circle (.4cm); 
	\node at (6.5,4.5) {$0$};
	\draw (6.5,4.5) circle (.4cm); 
	\node at (6.5,5.5) {$1$};
	\draw (6.5,5.5) circle (.4cm); 
	\node at (6.5,6.5) {$3$};
	\draw (6.5,6.5) circle (.4cm); 
	\node at (7.5,7.5) {$1$};
	\draw (7.5,7.5) circle (.4cm); 
	\filldraw (4,1) circle (2pt);
	\end{tikzpicture}\caption{Example of an element in $\PLSQE(2,6)^{\ast 1 }$}\label{fig: labelled square}
\end{figure}

\begin{remark}
	Observe that a partially labelled Dyck path is also a partially labelled square path, and indeed $\PLD(m,n)^{\ast k}\subset \PLSQE(m,n)^{\ast k}$.
\end{remark}

We define two statistics on this set that reduce to the same statistics as defined in \cite{Loehr-Warrington-square-2007} when $m=k=0$. 

\begin{definition}
	\label{def:sqarea}
	Let $P \in \PLSQE(m,n)^{\ast k}$ and $s$ be its shift. Define 
	\[
	\area(P) \coloneqq \sum_{i\not \in \DRise(P)} (a_i(P) + s).
	\] More visually, the area is the number of whole squares between the path and the base diagonal and not contained in rows containing a decorated rise. 
\end{definition}

For example, the path in Figure~\ref{fig: labelled square} has area $11$. 

\begin{definition} \label{def: dinv SQ}
	Let $P \in \PLSQE(m,n)$. For $1 \leq i < j \leq n+m$, we say that the pair $(i,j)$ is an \emph{inversion} if
	\begin{itemize}
		\item either $a_i(P) = a_j(P)$ and $l_i(P) < l_j(P)$ (\emph{primary inversion}),
		\item or $a_i(P) = a_j(P) + 1$ and $l_i(P) > l_j(P)$ (\emph{secondary inversion}),
	\end{itemize}
	where $l_i(P)$ denotes the label of the vertical step in the $i$-th row.
	
	Then we define 
	\begin{align*}
		\dinv(P) & \coloneqq \# \{ 0\leq i < j \leq n+m \mid (i,j) \; \text{is an inversion}\} \\
		& \quad + \#\{0\leq i\leq m+n \mid a_i(P)<0\text{ and } l_i(P)\neq 0 \}.
	\end{align*} 
	This second term is referred to as \emph{bonus dinv}. 
\end{definition}

For example, the path in Figure~\ref{fig: labelled square} has dinv $6$: $2$ primary inversions, i.e. $(1,7)$ and $(2,3)$, $1$ secondary inversion, i.e. $(1,6)$, and $3$ bonus dinv, coming from the rows $3$, $4$ and $6$. 

\begin{remark}
	Observe on partially labelled Dyck paths all our statistics agree with the statistics of the generalized Delta conjecture.
\end{remark}

\begin{definition}
	We define a formal series in the variables $\underline x=(x_1,x_2,\dots)$ and coefficients in $\mathbb N [q,t]$ 
	\[
	\gensquare \coloneqq \sum_{P\in \PLSQE(m,n)^{\ast k}} q^{\dinv(P)}t^{\area(P)}x^P. 
	\]
\end{definition}

In analogy with the Delta conjecture, we will refer to the case $m=0$ of the following conjecture simply as the \emph{Delta square conjecture}.

\begin{conjecture}[Generalized Delta square]
	For $m,n,k\in \mathbb{N}$, $m\geq 0$ and $n>k\geq 0$,
	\[
	\frac{[n-k]_t}{[n]_t} \Delta_{h_m}\Delta_{e_{n-k}}\omega(p_n) = \gensquare.
	\]	
\end{conjecture}
\begin{remark} \label{rmk:symmetry}
	Observe that $\gensquare$ is a symmetric function. Indeed, consider the expression $\sum_Pq^{\dinv(P)}t^{\area(P)}x^P$ where the sum is taken over all $P\in \PLSQE(m,n)^{\ast k}$ of a fixed \emph{shape}, i.e. a fixed underlying square path with prescribed zero valleys. From this sum we can factor $t^{\area(P)}$, as the area is the same for all such paths $P$, and $q^{a(P)}$, where $a(P)$ is the contribution to the dinv of the $0$ labels and of the negative letters of the area word (the bonus dinv): indeed this contribution does not depend on the nonzero labels, but only on the shape, so it will be the same for all our paths. What we are left with is in fact an LLT polynomial: the argument is essentially the same as in \cite{Haglund-Remmel-Wilson-2015}*{Section~6.2}, so we omit it (cf also \cite{Haglund-Book-2008}*{Remark~6.5}). As it is well-known that the LLT polynomials are symmetric functions (cf. \cite{Haglund-Haiman-Loehr}*{Appendix}), we deduce that also $\gensquare$ is symmetric.
\end{remark}

This conjecture answers a question in \cite{Haglund-Remmel-Wilson-2015}*{Section~8.2}.

\begin{remark}
	Notice that the case $m=k=0$ of the generalized Delta square conjecture reduces precisely to the \emph{square conjecture} of Loehr and Warrington \cite{Loehr-Warrington-square-2007}, recently proved by Sergel \cite{Leven-2016} after the breakthrough of Carlsson and Mellit \cite{Carlsson-Mellit-ShuffleConj-2015}.
\end{remark}
\begin{example}
	Using \eqref{eq:en_expansion} and \eqref{eq:p_expansion} (see Section~\ref{sec:SF} for definitions and notations about symmetric functions), it is easy to see that for $m=0$ and $k=n-1$ we get
	\begin{equation}
	\frac{[1]_t}{[n]_t} \Delta_{h_0}\Delta_{e_{1}}\omega(p_n) = [n]_qe_n.
	\end{equation}
	We leave to the reader the straightforward verification that indeed
	\begin{equation}
	\mathsf{PLSQ^E}_{\underline x,q,t}(0,n)^{\ast n-1} = [n]_qe_n,
	\end{equation}
	proving in this way our conjecture at $m=0$ and $k=n-1$.
\end{example}

\section{Symmetric functions} \label{sec:SF}

For all the undefined notations and the unproven identities, we refer to \cite{DAdderio-VandenWyngaerd-2017}*{Section~1}, where definitions, proofs and/or references can be found. In the next subsection we will limit ourselves to introduce some notation, while in the following one we will recall some identities that are going to be useful in the sequel. In the third and final subsection we will prove the main results on symmetric functions of this work.

For more references on symmetric functions cf. also \cite{Macdonald-Book-1995}, \cite{Stanley-Book-1999} and \cite{Haglund-Book-2008}.

\subsection{Notation}

We denote by $\Lambda=\bigoplus_{n\geq 0}\Lambda^{(n)}$ the graded algebras of symmetric functions with coefficients in $\mathbb{Q}(q,t)$, and by $\<\, , \>$ the \emph{Hall scalar product} on $\Lambda$, which can be defined by saying that the Schur functions form an orthonormal basis.

The standard bases of the symmetric functions that will appear in our
calculations are the monomial $\{m_\lambda\}_{\lambda}$, complete $\{h_{\lambda}\}_{\lambda}$, elementary $\{e_{\lambda}\}_{\lambda}$, power $\{p_{\lambda}\}_{\lambda}$ and Schur $\{s_{\lambda}\}_{\lambda}$ bases.

\emph{We will use implicitly the usual convention that $e_0 = h_0 = 1$ and $e_k = h_k = 0$ for $k < 0$.}

For a partition $\mu\vdash n$, we denote by
\begin{align}
	\widetilde{H}_{\mu} \coloneqq \widetilde{H}_{\mu}[X]=\widetilde{H}_{\mu}[X;q,t]=\sum_{\lambda\vdash n}\widetilde{K}_{\lambda \mu}(q,t)s_{\lambda}
\end{align}
the \emph{(modified) Macdonald polynomials}, where
\begin{align}
	\widetilde{K}_{\lambda \mu} \coloneqq \widetilde{K}_{\lambda \mu}(q,t)=K_{\lambda \mu}(q,1/t)t^{n(\mu)}\quad \text{ with }\quad n(\mu)=\sum_{i\geq 1}\mu_i(i-1)
\end{align}
are the \emph{(modified) Kostka coefficients} (see \cite{Haglund-Book-2008}*{Chapter~2} for more details). 

The set $\{\widetilde{H}_{\mu}[X;q,t]\}_{\mu}$ is a basis of the ring of symmetric functions $\Lambda$. This is a modification of the basis introduced by Macdonald \cite{Macdonald-Book-1995}.

If we identify the partition $\mu$ with its Ferrers diagram, i.e. with the collection of cells $\{(i,j)\mid 1\leq i\leq \mu_i, 1\leq j\leq \ell(\mu)\}$, then for each cell $c\in \mu$ we refer to the \emph{arm}, \emph{leg}, \emph{co-arm} and \emph{co-leg} (denoted respectively as $a_\mu(c), l_\mu(c), a_\mu(c)', l_\mu(c)'$) as the number of cells in $\mu$ that are strictly to the right, above, to the left and below $c$ in $\mu$, respectively.

We set $M \coloneqq (1-q)(1-t)$ and we define for every partition $\mu$
\begin{align}
	B_{\mu} & \coloneqq B_{\mu}(q,t)=\sum_{c\in \mu}q^{a_{\mu}'(c)}t^{l_{\mu}'(c)} \\
	T_{\mu} & \coloneqq T_{\mu}(q,t)=\prod_{c\in \mu}q^{a_{\mu}'(c)}t^{l_{\mu}'(c)} \\
	\Pi_{\mu} & \coloneqq \Pi_{\mu}(q,t)=\prod_{c\in \mu/(1)}(1-q^{a_{\mu}'(c)}t^{l_{\mu}'(c)}) \\
	w_{\mu} & \coloneqq w_{\mu}(q,t)=\prod_{c\in \mu} (q^{a_{\mu}(c)} - t^{l_{\mu}(c) + 1}) (t^{l_{\mu}(c)} - q^{a_{\mu}(c) + 1}).
\end{align}

We will make extensive use of the \emph{plethystic notation} (cf. \cite{Haglund-Book-2008}*{Chapter~1}).

We have for example the addition formulas
\begin{align}
	\label{eq:e_h_sum_alphabets}
	e_n[X+Y]=\sum_{i=0}^ne_{n-i}[X]e_i[Y]\quad \text{ and } \quad  h_n[X+Y]=\sum_{i=0}^nh_{n-i}[X]h_i[Y].
\end{align}
We will also use the symbol $\epsilon$ for
\begin{equation}
f[\epsilon X] = (-1)^k f[X]\qquad \text{ for } f[X]\in \Lambda^{(k)},
\end{equation}
so that, in general,
\begin{align}
	\label{eq:minusepsilon}
	f[-\epsilon X] = \omega f[X]
\end{align}
for any symmetric function $f$, where $\omega$ is the fundamental algebraic involution which sends $e_k$ to $h_k$, $s_{\lambda}$ to $s_{\lambda'}$ and $p_k$ to $(-1)^{k-1}p_k$.

Recall the \emph{Cauchy identities}
\begin{align}
	\label{eq:Cauchy_identities}
	h_n[XY] = \sum_{\lambda\vdash n} s_{\lambda}[X] s_{\lambda}[Y] \quad \text{ and } \quad h_n[XY] = \sum_{\lambda\vdash n} h_{\lambda}[X] m_{\lambda}[Y].
\end{align}

We will also use the \emph{star scalar product} on $\Lambda$, which can be defined for all $f,g\in \Lambda$ as
\begin{equation} \label{eq:starprod}
\langle f,g\rangle_* \coloneqq  \langle \omega \phi f,g\rangle=\langle \phi \omega f,g\rangle,
\end{equation}
where 
\begin{equation}
\phi f[X] \coloneqq f[MX]\qquad \text{ for all } f[X]\in \Lambda.
\end{equation}

It turns out that the Macdonald polynomials are orthogonal with respect to the star scalar product: more precisely
\begin{align}
	\label{eq:H_orthogonality}
	\< \widetilde{H}_{\lambda},\widetilde{H}_{\mu}\>_*=w_{\mu}(q,t)\delta_{\lambda, \mu}.
\end{align}

We define the \emph{nabla} operator on $\Lambda$ by
\begin{align}
	\nabla \widetilde{H}_{\mu} \coloneqq T_{\mu} \widetilde{H}_{\mu} \quad \text{ for all } \mu,
\end{align}
and we define the \emph{delta} operators $\Delta_f$ and $\Delta_f'$ on $\Lambda$ by
\begin{align}
	\Delta_f \widetilde{H}_{\mu} \coloneqq f[B_{\mu}(q,t)] \widetilde{H}_{\mu} \quad \text{ and } \quad 
	\Delta_f' \widetilde{H}_{\mu}  \coloneqq f[B_{\mu}(q,t)-1] \widetilde{H}_{\mu}, \quad \text{ for all } \mu.
\end{align}
Observe that on the vector space of symmetric functions homogeneous of degree $n$, denoted by $\Lambda^{(n)}$, the operator $\nabla$ equals $\Delta_{e_n}$. Moreover, for every $1\leq k\leq n$,
\begin{align}
	\label{eq:deltaprime}
	\Delta_{e_k} = \Delta_{e_k}' + \Delta_{e_{k-1}}' \quad \text{ on } \Lambda^{(n)},
\end{align}
and for any $k > n$, $\Delta_{e_k} = \Delta_{e_{k-1}}' = 0$ on $\Lambda^{(n)}$, so that $\Delta_{e_n}=\Delta_{e_{n-1}}'$ on $\Lambda^{(n)}$.

\medskip

For a given $k\geq 1$, we define the \emph{Pieri coefficients} $c_{\mu \nu}^{(k)}$ and $d_{\mu \nu}^{(k)}$ by setting
\begin{align}
	\label{eq:def_cmunu} h_{k}^\perp \widetilde{H}_{\mu}[X] & =\sum_{\nu \subset_k \mu} c_{\mu \nu}^{(k)} \widetilde{H}_{\nu}[X], \\
	\label{eq:def_dmunu} e_{k}\left[\frac{X}{M}\right] \widetilde{H}_{\nu}[X] & = \sum_{\mu \supset_k \nu} d_{\mu \nu}^{(k)} \widetilde{H}_{\mu}[X],
\end{align}
where $\nu\subset_k \mu$ means that $\nu$ is contained in $\mu$ (as Ferrers diagrams) and $\mu/\nu$ has $k$ lattice cells, and the symbol $\mu \supset_k \nu$ is analogously defined. The following identity is well-known:
\begin{align}
	\label{eq:rel_cmunu_dmunu} 
	c_{\mu \nu}^{(k)} = \frac{w_{\mu}}{w_{\nu}}d_{\mu \nu}^{(k)}.
\end{align}

The following summation formula is also well-known (e.g. cf. \cite{DAdderio-VandenWyngaerd-2017}*{Equation~1.35}):
\begin{equation} \label{eq:sumBmu}
\sum_{\nu\subset_1\mu}c_{\mu \nu}^{(1)} = B_\mu,
\end{equation}
while the following one is proved right after Equation~(5.4) in \cite{DAdderio-VandenWyngaerd-2017}: for $\alpha\vdash n$,
\begin{equation} \label{eq:summation}
\sum_{\nu\subset_\ell \alpha}c_{\alpha \nu}^{(\ell)}  T_\nu = e_{n-\ell}[B_\alpha].
\end{equation}

\medskip

We will also use the symmetric functions $E_{n,k}$, that were introduced in \cite{Garsia-Haglund-qtCatalan-2002} by means of the following expansion:
\begin{align}
	\label{eq:def_Enk}
	e_n \left[ X \frac{1-z}{1-q} \right] = \sum_{k=1}^n \frac{(z;q)_k}{(q;q)_k} E_{n,k},
\end{align}
where
\begin{align}
	(a;q)_s \coloneqq (1 - a)(1 - qa)(1 - q^2 a) \cdots (1 - q^{s-1} a)
\end{align}
is the usual $q$-\emph{rising factorial}.

Observe that 
\begin{equation} \label{eq:Enk}
e_n=\sum_{k=1}^nE_{n,k}.
\end{equation}

Recall also the standard notation for $q$-analogues: for $n, k\in \mathbb{N}$, we set
\begin{align}
	[0]_q \coloneqq 0, \quad \text{ and } \quad [n]_q \coloneqq \frac{1-q^n}{1-q} = 1+q+q^2+\cdots+q^{n-1} \quad \text{ for } n \geq 1,
\end{align}
\begin{align}
	[0]_q! \coloneqq 1 \quad \text{ and }\quad [n]_q! \coloneqq [n]_q[n-1]_q \cdots [2]_q [1]_q \quad \text{ for } n \geq 1,
\end{align}
and
\begin{align}
	\qbinom{n}{k}_q  \coloneqq \frac{[n]_q!}{[k]_q![n-k]_q!} \quad \text{ for } n \geq k \geq 0, \quad \text{ while } \quad \qbinom{n}{k}_q \coloneqq 0 \quad \text{ for } n < k.
\end{align}

Recall also (cf. \cite{Stanley-Book-1999}*{Theorem~7.21.2}) that
\begin{align}
	\label{eq:h_q_binomial}
	h_k[[n]_q] = \frac{(q^{n};q)_k}{(q;q)_k} = \qbinom{n+k-1}{k}_q \quad \text{ for } n \geq 1 \text{ and } k \geq 0,
\end{align}
and
\begin{align} \label{eq:e_q_binomial}
	e_k[[n]_q] = q^{\binom{k}{2}} \qbinom{n}{k}_q \quad \text{ for all } n, k \geq 0.
\end{align}

Moreover (cf. \cite{Stanley-Book-1999}{Corollary~7.21.3})
\begin{align}
	\label{eq:h_q_prspec}
	h_k\left[\frac{1}{1-q}\right] = \frac{1}{(q;q)_k} = \prod_{i=1}^k \frac{1}{1-q^i} \quad \text{ for } k \geq 0.
\end{align}

\subsection{Some basic identities}

First of all, we record the well-known
\begin{equation} \label{eq:Hn}
\widetilde{H}_{(n)}[X] = h_n\left[\frac{X}{1-q}\right] \prod_{i=1}^{n}(1-q^i),
\end{equation}
and the obvious
\begin{align} \label{eq:obvious}
	\notag	T_{(n)} & = q^{\binom{n}{2}}\\
	\notag B_{(n)} & = [n]_q\\
	\Pi_{(n)} & =\prod_{i=1}^{n}(1-q^i)\\
	\notag w_{(n)} & =\prod_{i=1}^{n}(1-q^i) \cdot \prod_{i=0}^{n-1}(q^i-t).
\end{align}

The following identity is well-known: for any symmetric function $f\in \Lambda^{(n)}$,
\begin{align}
	\label{eq:lem_e_h_Delta}
	\< \Delta_{e_{d}} f, h_n \> = \< f, e_d h_{n-d} \>.
\end{align}

We will use the following form of \emph{Macdonald-Koornwinder reciprocity}: for all partitions $\alpha$ and $\beta$
\begin{align}
	\label{eq:Macdonald_reciprocity}
	\frac{\widetilde{H}_{\alpha}[MB_{\beta}]}{\Pi_{\alpha}} = \frac{\widetilde{H}_{\beta}[MB_{\alpha}]}{\Pi_{\beta}}.
\end{align}
The following identity is also known as \emph{Cauchy identity}:
\begin{align}
	\label{eq:Mac_Cauchy}
	e_n \left[ \frac{XY}{M} \right] = \sum_{\mu \vdash n} \frac{ \widetilde{H}_{\mu} [X] \widetilde{H}_\mu [Y]}{w_\mu} \quad \text{ for all } n.
\end{align}

We need the following well-known proposition.
\begin{proposition} 
	For $n\in \mathbb{N}$ we have
	\begin{align}
		\label{eq:en_expansion}
		e_n[X] = e_n \left[ \frac{XM}{M} \right] = \sum_{\mu \vdash n} \frac{M B_\mu \Pi_{\mu} \widetilde{H}_\mu[X]}{w_\mu}.
	\end{align}
	Moreover, for all $k\in \mathbb{N}$ with $0\leq k\leq n$, we have
	\begin{align}
		\label{eq:e_h_expansion}
		h_k \left[ \frac{X}{M} \right] e_{n-k} \left[ \frac{X}{M} \right] = \sum_{\mu \vdash n} \frac{e_k[B_\mu] \widetilde{H}_\mu[X]}{w_\mu},
	\end{align}
	and
	\begin{align}
		\label{eq:p_expansion}
		\omega (p_n[X]) = [n]_q[n]_t\sum_{\mu \vdash n} \frac{M\Pi_\mu\widetilde{H}_\mu[X]}{w_\mu}.
	\end{align}
\end{proposition}

We will make use of the following easy proposition.
\begin{proposition}
	We have
	\begin{equation} \label{eq:deltaq0}
	{\nabla e_n}_{\big{|}_{t=0}} =\widetilde{H}_{(n)}[X;q,0] = \widetilde{H}_{(n)}[X;q,t] =h_n\left[\frac{X}{1-q}\right] \prod_{i=1}^{n}(1-q^i).
	\end{equation}	
\end{proposition}
\begin{proof}
	The result easily follows from \eqref{eq:Hn}, the expansion (cf. \eqref{eq:en_expansion})
	\begin{equation}
	\nabla e_n = \sum_{\mu\vdash n}T_\lambda\frac{M B_\mu \Pi_\mu \widetilde{H}_\mu[X]}{w_\mu},
	\end{equation}
	the obvious
	\begin{equation}
	T_\lambda(q,0)=\delta_{\lambda,(n)}T_{(n)}(q,0),
	\end{equation} 	
	and the identities \eqref{eq:obvious}.
\end{proof}

The following identity is \cite{Can-Loehr-2006}*{Theorem~4}:
\begin{equation} \label{eq:pn_Enk}
\omega(p_n)=\sum_{k=1}^n\frac{[n]_q}{[k]_q}E_{n,k}.
\end{equation}

\subsection{The family $F_{n,k;p}^{(d,\ell)}$}

Set
\begin{equation}
F_{n,k;p}^{(d,\ell)} \coloneqq t^{n-k-\ell} \< \Delta_{h_{n-k-\ell}} \Delta_{e_\ell} e_{n+p-d} \left[ X \frac{1-q^k}{1-q} \right], e_p h_{n-d} \>.
\end{equation}
We already considered this family in \cite{DAdderio-Iraci-VandenWyngaerd-GenDeltaSchroeder}*{Section~3.3}. We are going to recall here some of the results from that article.

The family of plethystic formulae $F_{n,k;p}^{(d,\ell)}$ satisfy the following recursion.

\begin{theorem}[\cite{DAdderio-Iraci-VandenWyngaerd-GenDeltaSchroeder}*{Corollary~3.5}]\label{thm: F iterated recursion}
	For $k,\ell,d,p\geq 0$, $n\geq k+\ell$ and $n+p\geq d$, the $F_{n,k;p}^{(d,\ell)}$ satisfy the following recursion: for $n\geq 1$ 
	\begin{equation}
	F_{n,n;p}^{(d,\ell)}=\delta_{\ell,0}q^{\binom{n-d}{2}}\qbinom{n}{n-d}\qbinom{n+p-1}{p}
	\end{equation}
	and, for $n\geq 1$ and $1\leq k<n$,
	\begin{align}
		F_{n,k;p}^{(d,\ell)} = & t^{n-k-\ell} \sum_{j=0}^{p} \sum_{s=0}^{k} q^{\binom{s}{2}} \qbinom{k}{s}_q \qbinom{k+j-1}{j}_q \\ 
		\notag	& \times  t^{p-j} \sum_{u=0}^{n-k-\ell} \sum_{v=0}^{s+j} q^{\binom{v}{2}} \qbinom{s+j}{v}_q \qbinom{s+j+u-1}{u}_q F_{n-k,u+v;p-j}^{(d-k+s,\ell-v)},
	\end{align}
	with initial conditions
	\begin{equation}
	F_{0,k;p}^{(d,\ell)}=\delta_{k,0}\delta_{p,0}\delta_{d,0}\delta_{\ell,0}\qquad \text{and} \qquad F_{n,0;p}^{(d,\ell)}=\delta_{n,0}\delta_{p,0}\delta_{d,0}\delta_{\ell,0}.
	\end{equation}
\end{theorem}

The $F_{n,k;p}^{(d,\ell)}$ can be rewritten in the following way.
\begin{lemma}[\cite{DAdderio-Iraci-VandenWyngaerd-GenDeltaSchroeder}*{Lemma~3.6}]
	For $k,\ell,d,p\geq 0$, $n\geq k+\ell$ and $n+p\geq d$, we have
	\begin{equation} \label{eq:lemnablaEnk}
	F_{n,k;p}^{(d,\ell)} = \sum_{\gamma\vdash n+p-d}\left.(\mathbf{\Pi}^{-1}\nabla E_{n-\ell,k}[X])\right|_{X=MB_\gamma} \frac{\Pi_\gamma}{w_\gamma}e_{\ell}[B_\gamma]e_p[B_\gamma],
	\end{equation}
	where $\mathbf{\Pi}$ is the invertible linear operator defined by
	\begin{equation} \label{eq:defPi}
	\mathbf{\Pi} \widetilde{H}_\mu[X] = \Pi_\mu \widetilde{H}_\mu[X] \qquad \text{ for all } \mu.
	\end{equation}
\end{lemma}

The interest in the $F_{n,k;p}^{(d,\ell)}$ lies in the following theorem.

\begin{theorem}[\cite{DAdderio-Iraci-VandenWyngaerd-GenDeltaSchroeder}*{Theorem~3.7}] \label{thm:recoFnkpdl}
	For $\ell,d,p\geq 0$, $n\geq \ell+1$ and $n\geq d$, we have
	$$
	\sum_{k=1}^{n-\ell}F_{n,k;p}^{(d,\ell)}= \< \Delta_{h_p}\Delta_{e_{n-\ell-1}}'e_n, e_{n-d}h_{d} \>.
	$$
\end{theorem}

\subsection{The family $S_{n,k;p}^{(d,\ell)}$}

Set 
\begin{equation}
S_{n,k;p}^{(d,\ell)} \coloneqq \frac{[n]_q}{[k]_q}	F_{n,k;p}^{(d,\ell)}.
\end{equation}

We have the following recursion for $S_{n,k;p}^{(d,\ell)}$.
\begin{theorem} \label{thm:Snkpdl_recursion}
	For $k,\ell,d,p\geq 0$, $n\geq k+\ell$ and $n\geq d$, the $S_{n,k;p}^{(d,\ell)}$ satisfy the following recursion: for $n\geq 1$ 
	\begin{equation}
	S_{n,n;p}^{(d,\ell)}=\delta_{\ell,0}q^{\binom{n-d}{2}}\qbinom{n}{n-d}\qbinom{n+p-1}{p}
	\end{equation}
	and, for $n\geq 1$ and $1\leq k<n$,
	\begin{align*}
		S_{n,k;p}^{(d,\ell)}& = F_{n,k;p}^{(d,\ell)}+q^kt^{n-\ell-k} \sum_{j=0}^p \sum_{s=0}^k q^{\binom{s}{2}}\qbinom{s+j}{s}_q\qbinom{k+j-1}{s+j-1}_q \times\\
		& \quad \times t^{p-j}\sum_{u=0}^{n-\ell-k} \sum_{v=0}^{s+j} q^{\binom{v}{2}}\qbinom{u+v}{v}_q \qbinom{s+j+u-1}{s+j-v}_q S_{n-k,u+v;p-j}^{(d-k+s,\ell-v)},
	\end{align*}
	with initial conditions
	\begin{equation}
	S_{0,k;p}^{(d,\ell)}=\delta_{k,0}\delta_{p,0}\delta_{d,0}\delta_{\ell,0}\qquad \text{and} \qquad S_{n,0;p}^{(d,\ell)}=\delta_{n,0}\delta_{p,0}\delta_{d,0}\delta_{\ell,0}.
	\end{equation}
\end{theorem}
\begin{proof}
	The first identity follows immediately from the corresponding one in Theorem~\ref{thm: F iterated recursion}.
	
	For the second one, using the obvious
	\begin{equation}
	\frac{[n]_q}{[k]_q}=\frac{[k]_q+q^k[n-k]_q}{[k]_q}=1+q^k\frac{[n-k]_q}{[k]_q},
	\end{equation}
	and the recursion Theorem~\ref{thm: F iterated recursion}, we get
	\begin{align*}
		S_{n,k;p}^{(d,\ell)} & = \frac{[n]_q}{[k]_q}F_{n,k;p}^{(d,\ell)}\\
		& =F_{n,k;p}^{(d,\ell)} + q^k\frac{[n-k]_q}{[k]_q}F_{n,k;p}^{(d,\ell)}\\
		& =F_{n,k;p}^{(d,\ell)} + q^k\frac{[n-k]_q}{[k]_q} t^{n-k-\ell} \sum_{j=0}^{p} \sum_{s=0}^{k} q^{\binom{s}{2}} \qbinom{k}{s}_q \qbinom{k+j-1}{j}_q \\ 
		\notag	& \times  t^{p-j} \sum_{u=0}^{n-k-\ell} \sum_{v=0}^{s+j} q^{\binom{v}{2}} \qbinom{s+j}{v}_q \qbinom{s+j+u-1}{u}_q F_{n-k,u+v;p-j}^{(d-k+s,\ell-v)}\\
		& =F_{n,k;p}^{(d,\ell)} + q^kt^{n-k-\ell} \sum_{j=0}^{p} \sum_{s=0}^{k} q^{\binom{s}{2}} \qbinom{s+j}{j}_q \qbinom{k+j-1}{s+j-1}_q \\ 
		\notag	& \times  t^{p-j} \sum_{u=0}^{n-k-\ell} \sum_{v=0}^{s+j} q^{\binom{v}{2}} \qbinom{u+v}{v}_q \qbinom{s+j+u-1}{s+j-v}_q S_{n-k,u+v;p-j}^{(d-k+s,\ell-v)},
	\end{align*}
	where in the last equality we just rearranged suitably the $q$-binomials. The initial conditions are easy to check.
\end{proof}

The interest in the $S_{n,k;p}^{(d,\ell)}$ lies in the following theorem.
\begin{theorem} \label{thm:sumSchroeder}
	For $\ell,d,p\geq 0$, $n\geq \ell+1$ and $n\geq d$, we have
	\begin{equation}
	\sum_{k=1}^{n-\ell}S_{n,k;p}^{(d,\ell)}=\frac{[n-\ell]_t}{[n]_t}\<\Delta_{h_p}\Delta_{e_{n-\ell}}\omega (p_{n}),e_{n-d}h_{d}\> .
	\end{equation}
\end{theorem}
\begin{proof}
	We have
	\begin{align*}
		\sum_{k=1}^{n-\ell}S_{n,k;p}^{(d,\ell)} & =\sum_{k=1}^{n-\ell}\frac{[n]_q}{[k]_q}	F_{n,k;p}^{(d,\ell)}\\
		\text{(using \eqref{eq:lemnablaEnk})}& =\sum_{k=1}^{n-\ell}\frac{[n]_q}{[k]_q}\sum_{\gamma\vdash n+p-d}\frac{\Pi_\gamma}{w_\gamma}\left.(\mathbf{\Pi}^{-1}\nabla E_{n-\ell,k})\right|_{X=MB_\gamma} e_\ell[B_\gamma]e_p[B_\gamma] \\
		\text{(using \eqref{eq:pn_Enk})}& = \frac{[n]_q}{[n-\ell]_q}\sum_{\gamma\vdash n+p-d}\frac{\Pi_\gamma}{w_\gamma}\left.(\mathbf{\Pi}^{-1}\nabla \omega(p_{n-\ell}))\right|_{X=MB_\gamma} e_\ell[B_\gamma]e_p[B_\gamma] \\
		\text{(using \eqref{eq:p_expansion})}& = [n]_q[n-\ell]_t\sum_{\gamma\vdash n+p-d}\sum_{\nu\vdash n-\ell}T_\nu M\frac{\widetilde{H}_\nu[MB_\gamma]}{w_\nu} \frac{\Pi_\gamma}{w_\gamma} e_\ell[B_\gamma]e_p[B_\gamma] \\
		\text{(using \eqref{eq:def_dmunu})}& = [n]_q[n-\ell]_t\sum_{\gamma\vdash n+p-d}\sum_{\nu\vdash n-\ell}T_\nu M\sum_{\alpha\supset_\ell \nu} d_{\alpha \nu}^{(\ell)}\frac{\widetilde{H}_\alpha[MB_\gamma]}{w_\nu} \frac{\Pi_\gamma}{w_\gamma} e_p[B_\gamma] \\
		\text{(using \eqref{eq:Macdonald_reciprocity})}& = [n]_q[n-\ell]_t\sum_{\alpha\vdash n}M\frac{\Pi_\alpha}{w_\alpha}\sum_{\gamma\vdash n+p-d} e_p[B_\gamma] \frac{\widetilde{H}_\gamma[MB_\alpha]}{w_\gamma}   \sum_{\nu\subset_\ell \alpha}c_{\alpha \nu}^{(\ell)}  T_\nu\\
		\text{(using \eqref{eq:e_h_expansion})}& = [n]_q[n-\ell]_t\sum_{\alpha\vdash n}M\frac{\Pi_\alpha}{w_\alpha}h_p[B_\alpha]e_{n-d}[B_\alpha]   \sum_{\nu\subset_\ell \alpha}c_{\alpha \nu}^{(\ell)}  T_\nu \\
		\text{(using \eqref{eq:summation})}& = [n]_q[n-\ell]_t\sum_{\alpha\vdash n}M\frac{\Pi_\alpha}{w_\alpha}h_p[B_\alpha]e_{n-d}[B_\alpha]  e_{n-\ell}[B_\alpha] \\
		\text{(using \eqref{eq:p_expansion})}& = \frac{[n-\ell]_t}{[n]_t}\<\Delta_{h_p}\Delta_{e_{n-\ell}}\omega (p_{n}),e_{n-d}h_{d}\> .
	\end{align*}
\end{proof}

\subsection{An interesting identity}

We start by proving the following theorem of symmetric functions. 
\begin{theorem} \label{thm:main_thm}
	Given $n\in \mathbb{N}$, $n\geq 1$ and $\lambda\vdash n$
	\begin{equation} \label{eq:main_thm}
	\sum_{s= 0}^{n-1}(-t)^s\Delta_{e_{n-s-1}}'s_{\lambda}= \left\{\begin{array}{cl} {\nabla e_n}_{\big{|}_{t=0}} \cdot (-t)^{k-1} & \text{ for }\lambda=(k,1^{n-k})\\
	0 & \text{ otherwise}  \end{array}\right.\,\, . 
	\end{equation}
\end{theorem}

In order to prove Theorem~\ref{thm:main_thm}, we need the following lemma.

\begin{lemma}
	Given $n\in \mathbb{N}$, $n\geq 1$ and $\mu\vdash n$
	\begin{equation} \label{eq:lemma}
	\sum_{s= 0}^{n}(-t)^s e_{n-s}[B_{\mu}]  = \left\{ \begin{array}{cl}
	\prod_{i=0}^{n-1}(q^i-t)&  \text{ if }\mu=(n)\\
	0 & \text{ otherwise}
	\end{array}\right.\,\, .
	\end{equation}	
\end{lemma}
\begin{proof}
	Observe that
	\begin{align*}
		\sum_{s= 0}^{n}(-t)^s e_{n-s}[B_\mu] & =\sum_{s= 0}^{n}(-1)^s h_s[t]e_{n-s}[B_\mu] \\
		\text{(using \eqref{eq:minusepsilon})}& =\sum_{s= 0}^{n} e_s[-t]  e_{n-s}[B_\mu]\\
		\text{(using \eqref{eq:e_h_sum_alphabets})}& = e_n[B_\mu-t]\, .
	\end{align*}
	Now if $(0,1)\in \mu$, then $B_\mu-t$ has $n-1$ positive monomial, so that $e_n[B_\mu-t]=0$. The only shape for which $(0,1)\notin \mu$ is $\mu=(n)$, for which 
	\begin{equation}
	\sum_{s= 0}^{n}(-t)^s e_{n-s}[B_{(n)}]=\sum_{s= 0}^{n}(-t)^s e_{n-s}[[n]_q]=\prod_{i=0}^{n-1}(q^i-t),
	\end{equation}
	where we used \eqref{eq:Hn}. 
\end{proof}
We are now ready to prove Theorem~\ref{thm:main_thm}.
\begin{proof}[Proof of Theorem~\ref{thm:main_thm}]
	First of all, using \eqref{eq:deltaprime}, it is easy to see that
	\begin{equation} 
	\sum_{s= 0}^{n}(-t)^s\Delta_{e_{n-s}}s_{\lambda} = (1-t)\sum_{s= 0}^{n-1}(-t)^s\Delta_{e_{n-s-1}}'s_{\lambda}.
	\end{equation}
	Now, using \eqref{eq:H_orthogonality}, we have
	\begin{align*}
		\sum_{s= 0}^{n}(-t)^s\Delta_{e_{n-s}}s_{\lambda} & = \sum_{s= 0}^{n}(-t)^s\Delta_{e_{n-s}}\sum_{\mu\vdash n}\langle s_{\lambda} , \widetilde{H}_{\mu}[X]\rangle_*\frac{\widetilde{H}_{\mu}[X]}{w_{\mu}}\\
		& = \sum_{\mu\vdash n}\langle s_{\lambda} , \widetilde{H}_{\mu}[X]\rangle_* \sum_{s= 0}^{n}(-t)^s e_{n-s}[B_\mu]\frac{\widetilde{H}_{\mu}[X]}{w_{\mu}}\\
		\text{(using \eqref{eq:lemma})}	& =  \langle s_{\lambda} , \widetilde{H}_{(n)}[X]\rangle_* \cdot  \prod_{i=0}^{n-1}(q^i-t)\frac{\widetilde{H}_{(n)}[X]}{w_{(n)}}\\
		\text{(using \eqref{eq:Hn} and \eqref{eq:obvious})}& = \langle s_{\lambda} , h_n\left[\frac{X}{1-q}\right] \prod_{i=1}^{n}(1-q^i) \rangle_* \cdot  h_n\left[\frac{X}{1-q}\right] \\
		\text{(using \eqref{eq:Cauchy_identities})}& = \langle s_{\lambda} , \sum_{\mu\vdash n} s_{\mu}[1-t]s_{\mu}\left[\frac{X}{M}\right]  \rangle_* \cdot   h_n\left[\frac{X}{1-q}\right] \prod_{i=1}^{n}(1-q^i)\\
		\text{(using \eqref{eq:starprod})}& = \sum_{\mu\vdash n} \langle s_{\lambda'} ,  s_{\mu}  \rangle  s_{\mu}[1-t]h_n\left[\frac{X}{1-q}\right] \prod_{i=1}^{n}(1-q^i)\\
		\text{(using \eqref{eq:Hn})}& = s_{\lambda'} [1-t]\widetilde{H}_{(n)}[X;q,0]\\
		\text{(using \eqref{eq:deltaq0})}& = s_{\lambda'} [1-t]{\nabla e_n}_{\big{|}_{t=0}}.
	\end{align*}
	
	Now we need the following well-known identity (see \cite{Garsia-Hicks-Stout-2011}*{Lemma~2.1}):
	for all $\mu\vdash n$ 
	\begin{equation} 
	s_{\mu}[1-u]=\left\{\begin{array}{ll}
	(-u)^r(1-u) & \text{ if }\mu=(n-r,1^r)\text{ for some }r\in \{0,1,2,\dots,n-1\},\\
	0 & \text{ otherwise}.
	\end{array}\right.
	\end{equation}
	Applying this one, we get
	\begin{align*}
		(1-t)\sum_{s= 0}^{n-1}(-t)^s\Delta_{e_{n-s-1}}'s_{\lambda} & =
		\sum_{s= 0}^{n}(-t)^s\Delta_{e_{n-s}}s_{\lambda}\\
		& = s_{\lambda'} [1-t]{\nabla e_n}_{\big{|}_{t=0}}=\left\{\begin{array}{ll}
			(-t)^{k-1}(1-t) & {\nabla e_n}_{\big{|}_{t=0}} \text{ if }\lambda=(k,1^{n-k})\\
			0 & \text{ otherwise}
		\end{array}\right. ,
	\end{align*}
	which is what we wanted to prove.
\end{proof}

\subsection{Some consequences}

We deduce some consequences of Theorem~\eqref{thm:main_thm}.

\begin{corollary} 
	Given $m,n\in \mathbb{N}$, $m\geq 0$, $n\geq 1$ and $\lambda\vdash n$
	\begin{equation} \label{eq:main_cor}
	\sum_{s= 0}^{n-1}(-t)^s\Delta_{h_m}\Delta_{e_{n-s-1}}'s_{\lambda}= \left\{\begin{array}{cl} \qbinom{m+n-1}{m}_q {\nabla e_n}_{\big{|}_{t=0}} \cdot (-t)^{k-1} & \text{ for }\lambda=(k,1^{n-k})\\
	0 & \text{ otherwise}  \end{array}\right.\,\, . 
	\end{equation}
\end{corollary}
\begin{proof}
	The result follows easily by applying the operator $\Delta_{h_m}$ to \eqref{eq:main_thm}, and using \eqref{eq:deltaq0} and \eqref{eq:h_q_binomial}. 
\end{proof}

The following two theorems have a nice combinatorial interpretation in terms of the Delta conjectures, that we are going to explain in Section~\ref{sec:Invo}.

Specializing \eqref{eq:main_cor} to $\lambda = (1^n)$, we get the following theorem.
\begin{theorem} \label{thm:SFinvoDelta}
	Given $m,n\in \mathbb{N}$, $m\geq 0$ and $n\geq 1$, we have
	\begin{equation} 
	\sum_{s= 0}^{n-1}(-t)^s\Delta_{h_m}\Delta_{e_{n-s-1}}'e_n=  \qbinom{m+n-1}{m}_q {\nabla e_n}_{\big{|}_{t=0}} \,\, . 
	\end{equation}
\end{theorem}

The following theorem is also an easy consequence of \eqref{eq:main_cor}.
\begin{theorem} \label{thm:SFinvoDeltaSQ}
	Given $m,n\in \mathbb{N}$, $m\geq 0$ and $n\geq 1$, we have
	\begin{equation} 
	\sum_{s= 0}^{n-1}(-t)^s\frac{[n-s]_t}{[n]_t}\Delta_{h_m}\Delta_{e_{n-s}}\omega(p_n)= \qbinom{m+n-1}{m}_q {\nabla e_n}_{\big{|}_{t=0}} \,\, . 
	\end{equation}
\end{theorem}
\begin{proof}
	Observe that
	\begin{equation} \label{eq:observ}
	\sum_{s= 0}^{n}(-1)^s \Delta_{e_{n-s}}f = 0 \qquad \text{ for any }f\in \Lambda^{(n)}
	\end{equation}
	since
	\begin{align*}
		\sum_{s= 0}^{n}(-1)^s \Delta_{e_{n-s}}f & =\sum_{s= 0}^{n}(-1)^s \Delta_{e_{n-s}} \sum_{\mu\vdash n} \<f,\widetilde{H}_\mu\>_* \frac{\widetilde{H}_\mu}{w_\mu}\\
		& = \sum_{\mu\vdash n} \<f,\widetilde{H}_\mu\>_* \sum_{s= 0}^{n}(-1)^s  e_{n-s}[B_\mu]\frac{\widetilde{H}_\mu}{w_\mu}\\
		& = \sum_{\mu\vdash n} \<f,\widetilde{H}_\mu\>_* e_{n}[B_\mu-1]\frac{\widetilde{H}_\mu}{w_\mu}\\
		& = 0.
	\end{align*}
	So, multiplying by $(1-t^n)$, we get
	\begin{align*}
		(1-t^n)\sum_{s= 0}^{n-1}(-t)^s\frac{[n-s]_t}{[n]_t}\Delta_{h_m}\Delta_{e_{n-s}}\omega(p_n)
		& = \sum_{s= 0}^{n}(-t)^s(1-t^{n-s}) \Delta_{h_m}\Delta_{e_{n-s}}\omega(p_n)\\
		& = \sum_{s= 0}^{n}(-t)^s \Delta_{h_m}\Delta_{e_{n-s}}\omega(p_n)\\
		& \quad + (-t^n) \sum_{s= 0}^{n}(-1)^{s}\Delta_{h_m}\Delta_{e_{n-s}}\omega(p_n)\\
		\text{(using \eqref{eq:observ})}& = \sum_{s= 0}^{n}(-t)^s \Delta_{h_m}\Delta_{e_{n-s}}\omega(p_n)\\
		\text{(using \eqref{eq:deltaprime})}& = (1-t)\sum_{s= 0}^{n-1}(-t)^s \Delta_{h_m}\Delta_{e_{n-s-1}}'\omega(p_n)\\
		& = (1-t)\sum_{s= 0}^{n-1}(-t)^s \Delta_{h_m}\Delta_{e_{n-s-1}}'(-1)^{n-1}\sum_{r=0}^{n-1}(-1)^rs_{(n-r,1^r)}\\
		\text{(using \eqref{eq:main_cor})}& = (-1)^{n-1}(1-t) \sum_{r=0}^{n-1}(-1)^r  \qbinom{m+n-1}{m}_q {\nabla e_n}_{\big{|}_{t=0}} \cdot (-t)^{n-r-1}\\
		& = (1-t^n) \qbinom{m+n-1}{m}_q {\nabla e_n}_{\big{|}_{t=0}},
	\end{align*}
	where in the fifth equality we used Murnaghan-Nakayama rule.
\end{proof}

\section{The Delta square at $q=0$}

In this section we prove the Delta square conjecture (i.e. $m=0$) at $q=0$.
\begin{theorem}
	For $n,k\in \mathbb{N}$, $n>k\geq 0$,
	\begin{equation}
	\frac{[n-k]_t}{[n]_t} \left.\Delta_{e_{n-k}}\omega(p_n)\right|_{q=0} = \mathsf{PLSQ^E}_{\underline x,0,t}(0,n)^{\ast k}.
	\end{equation}
\end{theorem}
\begin{proof}
	Looking at the combinatorial side, we observe that setting $q=0$ leaves out only labelled square paths with dinv $0$: because of the bonus dinv, this means that we are left with the partially labelled Dyck paths of dinv $0$, i.e.
	\begin{equation}
	\mathsf{PLSQ^E}_{\underline x,0,t}(0,n)^{\ast k} = \mathsf{PLD}_{\underline x,0,t}(0,n)^{\ast k}.
	\end{equation}
	But the Delta conjecture at $q=0$ has been proved in \cite{Garsia-Haglund-Remmel-Yoo-2017}, so we already know that
	\begin{equation} \label{eq:deltaq0_orig}
	\left.\Delta_{e_{n-k-1}}'e_n\right|_{q=0}  = \mathsf{PLD}_{\underline x,0,t}(0,n)^{\ast k}.
	\end{equation}
	\begin{remark}
		We will give a new proof of \eqref{eq:deltaq0_orig} in the Appendix. Unlike the proof in \cite{Garsia-Haglund-Remmel-Yoo-2017} (or even the alternative proof appearing in \cite{Haglund-Rhoades-Shimozono-arxiv}), where, using the symmetry in $q$ and $t$, they work combinatorially with $\mathsf{PLD}_{\underline x,q,0}(0,n)^{\ast k}$, in our new proof we will work directly with $\mathsf{PLD}_{\underline x,0,t}(0,n)^{\ast k}$. Morever, the symmetric function side of the proof is completely new.
	\end{remark}
	
	Therefore, in order to prove our theorem, it is enough to show that
	\begin{equation}
	\left.\Delta_{e_{n-k-1}}'e_n\right|_{q=0}  = \frac{[n-k]_t}{[n]_t} \left.\Delta_{e_{n-k}}\omega(p_n)\right|_{q=0}.
	\end{equation}
	But observe that for any partition $\mu\vdash n$
	\begin{align*}
		\left. e_{n-k-1}[B_\mu-1]\cdot B_\mu \right|_{q=0} & = t^{\binom{n-k}{2}}\qbinom{\ell(\mu)-1}{n-k-1}_t[\ell(\mu)]_t\\
		& = t^{\binom{n-k}{2}}\qbinom{\ell(\mu)}{n-k}_t [n-k]_t\\
		& = [n-k]_t \left. e_{n-k}[B_\mu] \right|_{q=0} ,
	\end{align*}
	where we used the easy specialization (compare \eqref{eq:e_q_binomial})
	\begin{equation} \label{eq:espec}
	\left. e_{n-k-1}[B_\mu-1] \right|_{q=0} = t^{\binom{n-k}{2}} \qbinom{\ell(\mu)-1}{n-k-1}_t.
	\end{equation}
	So, using \eqref{eq:en_expansion} and \eqref{eq:p_expansion}, we get
	\begin{align*}
		\left.\Delta_{e_{n-k-1}}'e_n\right|_{q=0} & = \sum_{\mu\vdash n}e_{n-k-1}[B_\mu-1]B_\mu\left.\frac{M\Pi_\mu\widetilde{H}_\mu[X]}{w_\mu}\right|_{q=0} \\
		& = [n-k]_t\sum_{\mu\vdash n}e_{n-k}[B_\mu] \left.\frac{M\Pi_\mu\widetilde{H}_\mu[X]}{w_\mu}\right|_{q=0} \\
		&= \frac{[n-k]_t}{[n]_t} \left.\Delta_{e_{n-k}}\omega(p_n)\right|_{q=0},
	\end{align*}
	as we wanted.
\end{proof}

\section{The generalized Delta at $k=0$ and $t=0$}

In this section we prove the generalized Delta conjecture at $k=0$ and $t=0$.
\begin{proposition} \label{prop:t0}
	For $n\in \mathbb{N}$, $n >0$,
	\begin{equation}
	\left.\Delta_{h_m} \Delta_{e_{n-1}}'e_n\right|_{t=0} =\left.\Delta_{h_m} \nabla e_n\right|_{t=0} = \mathsf{PLD}_{\underline x,q,0}(m,n)^{\ast 0}.
	\end{equation}
\end{proposition}
\begin{proof}
	Using \eqref{eq:deltaq0}, we have
	\begin{align*}
		\Delta_{h_m}{\nabla e_n}_{\big{|}_{t=0}} & =h_m[[n]_q]\widetilde{H}_{(n)}[X;q,0]\\
		& =\qbinom{n+m-1}{m}_q h_n\left[\frac{X}{1-q}\right] \prod_{i=1}^{n}(1-q^i)\\
		\text{(using \eqref{eq:Cauchy_identities})}& =\qbinom{n+m-1}{m}_q \sum_{\lambda\vdash n}\prod_{i=1}^{n}(1-q^i) h_{\lambda}\left[\frac{1}{1-q}\right] m_{\lambda}\left[X\right] \\
		\text{(using \eqref{eq:h_q_prspec})}& =\qbinom{n+m-1}{m}_q \sum_{\lambda\vdash n}\qbinom{n}{\lambda_1,\dots,\lambda_{\ell(\lambda)}}_q m_{\lambda}\left[X\right] .
	\end{align*}
	It is a well-known theorem of MacMahon (cf. \cite{Loehr-Book}*{Theorem~6.44}) that
	\begin{equation}
	\sum_{\lambda\vdash n}\qbinom{n}{\lambda_1,\dots,\lambda_{\ell(\lambda)}}_q m_{\lambda}\left[X\right] = \sum_{w\in \mathbb{P}^n}q^{inv(w)}x^w
	\end{equation}
	where $\mathbb{P}\coloneqq\{1,2,\dots\}$, $inv(w)$ is the number of inversions of the word $w\in \mathbb{P}^n$, and $x^w$ is defined as $x^w\coloneqq \prod_{i=1}^nx_i^{\text{number of $i$ in $w$}}$. Now at $t=0$, i.e. with area $0$, a labelled Dyck path of size $n$ reduces to a word in $\mathbb{P}^n$ (read top to bottom along the base diagonal), and its dinv is precisely the number of inversions of this word, so
	\begin{equation}
	\sum_{w\in \mathbb{P}^n}q^{inv(w)}x^w = \mathsf{PLD}_{\underline x,q,0}(0,n)^{\ast 0}.
	\end{equation}
	Now for each element of $\mathsf{PDL}(0,n)^{\ast 0}$ we can insert $m$ zero valleys in all possible ways, except in the lowest row, to get an element in $\mathsf{PDL}(m,n)^{\ast 0}$, and all the elements in $\mathsf{PDL}(m,n)^{\ast 0}$ are obtained in this way. Taking into account the contribution of the zero valleys to the dinv explains the factor $\qbinom{n+m-1}{m}_q$, so that
	\begin{equation}
	\qbinom{n+m-1}{m}_q\sum_{w\in \mathbb{P}^n}q^{inv(w)}x^w = \qbinom{n+m-1}{m}_q\mathsf{PLD}_{\underline x,q,0}(0,n)^{\ast 0}=\mathsf{PLD}_{\underline x,q,0}(m,n)^{\ast 0},
	\end{equation}
	completing the proof.
\end{proof}

\section{The Schr\"oder case}

The following definitions extend the ones in \cite{DAdderio-Iraci-VandenWyngaerd-GenDeltaSchroeder}*{Section~4} from Dyck paths to square paths (ending east).

\begin{definition}
	We define the set of \emph{valleys} of a square path $P \in \SQE(n)$.
	\begin{itemize}
		\item If $P$ starts with a north step
		\begin{align*}
			\Val(P) \coloneqq & \; \{2\leq i\leq n \mid a_i(P)\leq a_{i-1}(P)\},
		\end{align*}
		\item If $P$ starts with an east step 
		\begin{align*}
			\Val(P) \coloneqq & \; \{1\} \cup \{2\leq i\leq n \mid a_i(P)\leq a_{i-1}(P)\}.
		\end{align*}
	\end{itemize}
	These are exactly the indices of vertical steps that are directly preceded by a horizontal step.
\end{definition}

\begin{definition}
	The \emph{peaks} of a square path $P \in \SQE(n)$ are \[ \Peak(P) \coloneqq \{1 \leq i \leq n-1 \mid a_{i+1}(P) \leq a_i(P) \} \cup \{n\}, \] or the indices of vertical steps that are followed by a horizontal step.
\end{definition}

\begin{definition}\label{def: schroder_objects}
	Fix $p, n, \ell, d \in \mathbb{N}$, $n \geq 1$. For every square path $P \in \SQE(n+p)$ with $|\Rise(P)|\geq \ell$, $|\Peak(P)|\geq d$ and $|\Val(P)|\geq p$ choose three subsets of $\{1,\dots, n+p\}$:
	\begin{enumerate}[(i)]
		\item $\DRise(P) \subseteq \Rise(P)$ (see Definition \ref{def: rise}) such that $\vert \DRise(P) \vert = \ell$ and decorate the corresponding vertical steps with a $\ast$. 
		\item $\DPeak(P)\subseteq \Peak(P)$ such that $\vert \DPeak(P) \vert = d$ and decorate with a $\textcolor{blue}{\bullet}$ the points joining these vertical steps with the horizontal steps following them. We will call these \emph{decorated peaks}. 
		\item $\ZVal(P)\subseteq \Val(P)$ such that $\vert \ZVal \vert = p$ and $\DPeak(P) \cap \ZVal(P) = \emptyset$. Furthermore, if \[S\coloneqq \{1\leq i \leq n+p \mid a_i(P)= -s \},\] where $s$ is the shift of $P$; then $S\not \subseteq \ZVal(P)$. In other words, there exists at least one vertical step starting from the base diagonal that is not in $\ZVal(P)$.  Label the corresponding vertical steps with a zero. These steps will be called \emph{zero valleys}.
	\end{enumerate} 
	We denote the set of these paths by $\SQE(p,n)^{\ast \ell, \circ d}$. See Figure~\ref{fig:decsqe} for an example. 
\end{definition}

We define two statistics on $\SQE(p,n)^{\ast \ell, \circ d}$.

The definition of the \emph{area} of a path in $\SQE(p,n)^{\ast \ell, \circ d}$ is the same for a path in $\PLSQE(p,n)^{\ast \ell}$ (see Definition \ref{def:sqarea}).

\begin{definition}
	\label{def: dinv unlabelled SQ}
	For $P \in \SQE(p,n)^{\ast \ell, \circ d}$, and $1 \leq i<j \leq n+p$, we say that the pair $(i,j)$ is an inversion if 
	\begin{itemize}
		\item either $a_i(P) = a_j(P)$, $i \not \in \DPeak(P)$, and $j \not\in \ZVal(P)$ (\emph{primary inversion}),
		\item or $a_i(P) = a_j(P) + 1$, $j \not\in \DPeak(P)$, and  $i \not \in \ZVal(P)$ (\emph{secondary inversion}).
	\end{itemize}
	Then we define 
	\begin{align*}
		\dinv(P) & \coloneqq \# \{ 0\leq i < j \leq n+p \mid (i,j) \; \text{is an inversion}\} \\
		& \quad + \#\{0\leq i\leq n+p \mid a_i(D)<0\text{ and } i \not\in \ZVal(P) \}.
	\end{align*} 
	This second term is referred to as \emph{bonus dinv}.
\end{definition}

\begin{figure}[ht]
	\begin{tikzpicture}[scale = 0.8]
	\draw[step=1.0, gray!60, thin] (0,0) grid (8,8);
	
	\draw[gray!60, thin] (3,0) -- (8,5);
	
	\draw[blue!60, line width=1.6pt] (0,0) -- (0,1) -- (1,1) -- (2,1) -- (3,1) -- (4,1) -- (4,2) -- (5,2) -- (5,3) -- (5,4) -- (6,4) -- (6,5) -- (6,6) -- (6,7) -- (7,7) -- (7,8) -- (8,8);
	
	\node at (5.5,5.5) {$\ast$};
	
	\node at (4.5,1.5) {$0$};
	\draw (4.5,1.5) circle (.4cm); 
	\node at (6.5,4.5) {$0$};
	\draw (6.5,4.5) circle (.4cm); 
	
	\filldraw (4,1) circle (2pt);
	\filldraw[fill=blue!60] (6,7) circle (3 pt);
	\end{tikzpicture}\caption{Example of an element in $\SQE(2,6)^{\ast 1, \circ 1}$}
	\label{fig:decsqe}
\end{figure}
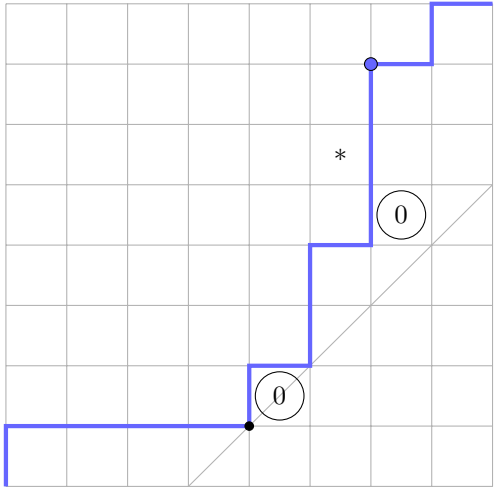

For example, the path in Figure~\ref{fig:decsqe} has dinv 7: 3 primary inversions, i.e. $(1,7)$, $(1,8)$ and $(2,3)$, 1 secondary inversion, i.e. $(1,6)$, and 3 bonus dinv, coming from the rows $3$, $4$ and $6$. 

\begin{remark} \label{rmk:dinv_reading}
	Let $P \in \PLSQE(p,n)$. We define its \emph{dinv reading word} as the sequence of labels read starting from the ones on the base diagonal $y=x-s$ (so that $s$ is the shift of $P$) going bottom to top, left to right; next the ones in the diagonal $y=x-s+1$ bottom to top, left to right; then the ones in the diagonal $y=x-s+2$ and so on. For example the path in Figure~\ref{fig:decsqe-labelled} has dinv reading word $01203465$.
	
	One can consider the paths in $\SQE(p,n)^{\ast \ell, \circ d}$ as partially labelled decorated square paths where the reading word is a shuffle of $p$ $0$'s, the string $1, \cdots, n-d$, and the string $n, \cdots n-d+1$. Indeed, given this restriction and the information about the position of the zero labels and considering the $d$ biggest labels to label the decorated peaks, the rest of the labelling is fixed. With regard to this labelling the Definitions \ref{def: dinv unlabelled SQ} and \ref{def: dinv SQ} of the dinv coincide.
	
	\begin{figure}[!ht]
		
		\begin{tikzpicture}[scale = 0.8]
		\draw[step=1.0, gray!60, thin] (0,0) grid (8,8);
		
		\draw[gray!60, thin] (3,0) -- (8,5);
		
		\draw[blue!60, line width=1.6pt] (0,0) -- (0,1) -- (1,1) -- (2,1) -- (3,1) -- (4,1) -- (4,2) -- (5,2) -- (5,3) -- (5,4) -- (6,4) -- (6,5) -- (6,6) -- (6,7) -- (7,7) -- (7,8) -- (8,8);
		
		\node at (5.5,5.5) {$\ast$};
		
		\node at (0.5,0.5) {$4$};
		\draw (0.5,0.5) circle (.4cm); 
		\node at (4.5,1.5) {$0$};
		\draw (4.5,1.5) circle (.4cm); 
		\node at (5.5,2.5) {$1$};
		\draw (5.5,2.5) circle (.4cm); 
		\node at (5.5,3.5) {$2$};
		\draw (5.5,3.5) circle (.4cm); 
		\node at (6.5,4.5) {$0$};
		\draw (6.5,4.5) circle (.4cm); 
		\node at (6.5,5.5) {$3$};
		\draw (6.5,5.5) circle (.4cm); 
		\node at (6.5,6.5) {$6$};
		\draw (6.5,6.5) circle (.4cm); 
		\node at (7.5,7.5) {$5$};
		\draw (7.5,7.5) circle (.4cm); 
		
		\filldraw (4,1) circle (2pt);
		\end{tikzpicture}
		
		\caption{Partially labelled square path corresponding to the example in Figure~\ref{fig:decsqe} .}	
		\label{fig:decsqe-labelled}
		
	\end{figure}
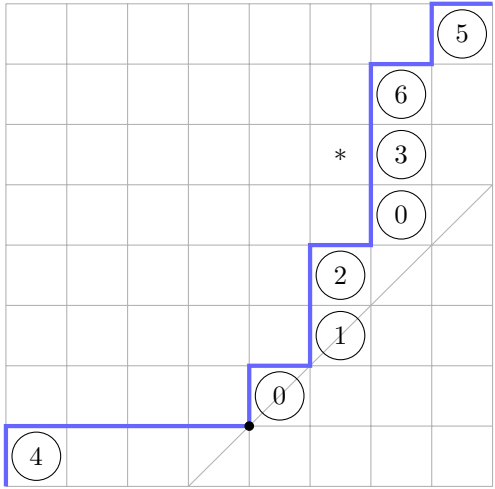
	
	For example, the path in Figure~\ref{fig:decsqe-labelled} is the partially labelled square path corresponding to the decorated square path in Figure~\ref{fig:decsqe}. Indeed it has dinv reading word $01203465$ which is a shuffle of two $0$'s and the strings $\{1,2,3,4,5\}$, $\{6\}$. Its dinv equals 7: 3 primary plus 1 secondary plus 3 bonus. 
\end{remark}

Define the subset 
\begin{equation}
\SQE(p, n \backslash k)^{\ast \ell, \circ d} \subseteq \SQE(p,n)^{\ast \ell, \circ d}
\end{equation} 
to consist of the paths $P \in \SQE(p,n)^{\ast \ell, \circ d}$ such that \[\# \{1 \leq i \leq n \mid a_i(P) \,\text{is minimum and} \, i \not \in \ZVal(D) \} = k, \]
and set 
\begin{equation}
\SQE_{q,t}(p, n \backslash k)^{\ast \ell, \circ d}\coloneqq \sum_{P \in \SQE(p, n \backslash k)^{\ast \ell, \circ d}} q^{\dinv(P)} t^{\area(P)} 
\end{equation}

Following \cite{DAdderio-Iraci-VandenWyngaerd-GenDeltaSchroeder}*{Section~4}, we denote by $\mathsf{DD}(p,n)^{\ast \ell, \circ d}$ the subset of $\SQE(p,n)^{\ast \ell, \circ d}$ consisting of the elements whose underlying path is a Dyck path (i.e. the minimum of the area word is $0$), and we set
\begin{equation}
\mathsf{DDd}(p,n\backslash k)^{\ast \ell, \circ d}:= \SQE(p, n \backslash k)^{\ast \ell, \circ d}\cap \mathsf{DD}(p,n)^{\ast \ell, \circ d},
\end{equation}
and
\begin{equation}
\mathsf{DDd}_{q,t}(p,n\backslash k)^{\ast \ell, \circ d}:= \sum_{D\in \mathsf{DDd}(p,n\backslash k)^{\ast \ell, \circ d}}q^{\dinv(D)}t^{\area(D)}.
\end{equation}

We recall here the main result from \cite{DAdderio-Iraci-VandenWyngaerd-GenDeltaSchroeder}.
\begin{theorem}[\cite{DAdderio-Iraci-VandenWyngaerd-GenDeltaSchroeder}*{Theorem~4.7}] \label{thm:Fnkp_Theorem}
	$\mathsf{DDd}_{q,t}(p,n\backslash k)^{\ast \ell, \circ d} = F^{(d,\ell)}_{n,k; p}$.
\end{theorem}

We are going the prove the analogue of the above theorem for square paths (ending east).

\begin{theorem}
	\label{thm:square_recursion}
	$\SQE_{q,t}(p, n \backslash k)^{\ast \ell, \circ d} = S^{(d,\ell)}_{n,k; p}$ .
\end{theorem}

\begin{proof}
	We will show that $\SQE_{q,t}(p, n \backslash k)^{\ast \ell, \circ d}$ satisfies the same recursion and initial conditions as $S^{(d,\ell)}_{n,k; p}$ in Theorem~\ref{thm:Snkpdl_recursion}.
	
	In other words we will show that 
	\begin{align*}
		\SQE_{q,t} & (p, n \backslash k)^{\ast \ell, \circ d} = F^{(d,\ell)}_{n,k; p} + q^k t^{n-k-\ell} \sum_{j=0}^p \sum_{s=0}^k q^{\binom{s}{2}} \qbinom{s+j}{s}_q \qbinom{k+j-1}{s+j-1}_q \\
		& \times t^{p-j} \sum_{u=0}^{n-k-\ell} \sum_{v=0}^{s+j} q^{\binom{v}{2}} \qbinom{u+v}{v}_q \qbinom{s+j+u-1}{s+j-v}_q \SQE_{q,t}(p-j, n-k \backslash u+v)^{\ast \ell-v, \circ d-(k-s)}
	\end{align*}
	with \[ \SQE_{q,t}(p, n \backslash n) = F^{(d,\ell)}_{n,n; p} = \delta_{\ell,0}q^{\binom{n-d}{2}}\qbinom{n}{n-d}\qbinom{n+p-1}{p}. \]
	
	The last identity is straightforward: if all the letters of the area word that are not zero valleys are minima, since the condition of ending east implies that one of them must be on the main diagonal (i.e. the corresponding letter of the area word is $0$), then all of them are on the main diagonal, hence the minimum of the area word is $0$ and the path is actually a Dyck path. The identity then follows from Theorem~\ref{thm: F iterated recursion}. Note here that we use the fact that $1$ is not a valley if the path starts with a north step. 
	
	Now for the recursive step. We give an overview of the combinatorial interpretations of all the variables appearing in this formula. We say that a vertical step of a path is \emph{at height $i$} if its corresponding letter in the area word equals $m+i$, where $m$ is the minimum of the area word (i.e. the steps on the base diagonal are at height $0$).
	
	\begin{itemize}
		
		\item $k-s$ is the number of decorated peaks at height $0$.
		\item $s$ is the number of minima in the area word whose index is not a decorated peak nor a zero valley.
		\item $j$ is the number of zero valleys at height $0$.
		\item $v$ is the number of decorated rises at height $1$.
		\item $u+v$ is the number of $m+1$'s in the area word whose index is not a zero valley. 
	\end{itemize}
	
	Start from a path $P$ in $\SQE(p, n \backslash k)^{\ast \ell, \circ d}$. If it is a Dyck path, thanks to \ref{thm:Fnkp_Theorem} it is counted by $F^{(d,\ell)}_{n,k; p}$.  Otherwise, remove all the minima from the area word, and then remove both the corresponding decoration on peaks, and decorations on rises at height one (which are not rises any more). In this way we obtain a path in  \[ \SQE(p-j, n-k \backslash u+v)^{\ast \ell-v, \circ d-(k-s)}. \]
	
	Notice that the steps we are deleting from the path in this way never lie on the line $x=y$ because the path is not a Dyck path. This implies that we do not need to make a distinction between paths starting north or east, i.e. paths where $1$ is a valley or not. Indeed all the vertical steps at height 0 are allowed to be zero valleys and the zero valleys at height 1 do not create any secondary dinv with the deleted letters since they are zero valleys. 
	
	Let us look at what happens to the statistics of the path. 
	
	The area goes down by the size ($n+p$), minus the number of zeroes in the area word ($k+j$) and the number of rises ($\ell$). This explains the term $t^{n-k-\ell} \cdot t^{p-j}$.
	
	The factor $q^k$ takes into account the bonus dinv that the minima generated (being them negative letters and not zero valleys). The factor $q^{\binom{s}{2}}$ takes into account the primary dinv among the minima that are neither zero valleys nor decorated peaks. The factor $\qbinom{s+j}{s}_q$ takes into account the primary dinv among the minima that are neither zero valleys nor decorated peaks, and minima that are zero valleys. Indeed, each time a one of the former follows one of the latter one unit of primary dinv is created. The factor $\qbinom{k+j-1}{s+j-1}_q$ takes into account the primary dinv among the minima that are decorated peaks (which are $k-s$) and the other minima (which are $s+j$), where we get $s+j-1$ because the last minimum cannot be a peak (thus it can't be a decorated peak).
	
	The factor $q^{\binom{v}{2}}$ takes into account the secondary dinv among steps at height $1$ that are decorated rises and steps at height $0$ that are directly below a decorated rise. The factor $\qbinom{u+v}{v}_q$ takes into account the secondary dinv among labels at height $1$ that are neither decorated rises nor zero valleys, and labels below a decorated rise. The factor $\qbinom{s+j+u-1}{s+j-v}_q$ takes into account the secondary among all the labels at height $1$ that are not zero valleys (which are $u+v$), and the labels at height $0$ that are neither decorated peaks nor below a decorated rise (which are $s+j-v$), where we get $u+v-1$ because the last rise comes after all the minima (because the last letter of the area word is non-negative).
	
	Summing over all the possible values of $j$, $s$, $u$, and $v$, we obtain the stated recursion. The initial conditions are easy to check.
\end{proof}

Since at least one of the steps at height 0 is not zero valley (see Definition~\ref{def: schroder_objects}), $k$ has to be at least $1$ and we get

\begin{equation}
\sum_{k=1}^{n-\ell} \SQE_{q,t}(p, n \backslash k)^{\ast \ell, \circ d} =\SQE_{q,t}(p, n)^{\ast \ell, \circ d}.
\end{equation}

Combining this with Theorem~\ref{thm:sumSchroeder} we deduce the Schr\"{o}der case of our generalized Delta square conjecture.

\begin{theorem}
	For $n,\ell,d,p\in \mathbb{N}$, $p\geq 0$, $n>\ell\geq 0$ and $n\geq d\geq 0$,
	\begin{equation}
	\frac{[n-\ell]_t}{[n]_t}\<\Delta_{h_p}\Delta_{e_{n-\ell}}\omega (p_{n}),e_{n-d}h_{d}\> = \SQE_{q,t}(p, n)^{\ast \ell, \circ d}.
	\end{equation}
\end{theorem}

Notice that the $q,t$-square theorem of Can and Loehr \cite{Can-Loehr-2006} is the special case $p=\ell=d=0$ of our theorem.

\section{An involution}\label{sec:Invo}

Fix $m,n\in \mathbb{N}$, $m\geq 0$ and $n>0$. Let 
\begin{equation}
X\coloneqq\bigsqcup_{k=0}^{n-1}\mathsf{PLSQ^E}(m,n)^{\ast k},
\end{equation}
and define a map $\varphi \colon X \to X$ in the following way: if $P\in X$ has no rises, i.e. no two consecutive vertical steps, then $\varphi(P) \coloneqq P$; otherwise, consider the first rise encountered by following the path $P$ starting from its breaking point (notice that this rise will always occur before the north-east corner): if the rise is decorated/undecorated, then $\varphi(P)$ is the path obtained from $P$ by undecorating/decorating that rise. Observe that $\varphi$ is clearly an involution, whose fixed points are the paths $P\in X$ with no rises, i.e. the paths of area $0$ with no decorated rises. Notice also that $\varphi$ restricts to an involution of
\begin{equation}
Y\coloneqq\bigsqcup_{k=0}^{n-1}\mathsf{PLD}(m,n)^{\ast k}\subset X.
\end{equation}

For any $P\in X$ we define a \emph{weight} by setting
\begin{equation}
wt(P)\coloneqq(-t)^{\mathsf{dr}(P)}q^{\dinv(P)}t^{\area(P)}x^P
\end{equation}
where $\mathsf{dr}(P)$ is defined to be the number of decorated rises of $P$.

Observe that
\begin{equation}
\sum_{P \in X}wt(P) = \sum_{s=0}^{n-1}(-t)^s\mathsf{PLSQ^E}_{\underline{x},q,t}(m,n)^{\ast s}
\end{equation}
and
\begin{equation}
\sum_{P \in Y}wt(P) = \sum_{s=0}^{n-1}(-t)^s\mathsf{PLD}_{\underline{x},q,t}(m,n)^{\ast s}.
\end{equation}

Suppose that $P\in X$ is such that $\varphi(P)\neq P$. Notice that the rise occurring in the definition of $\varphi$ is always at distance $1$ from the base diagonal, so undecorating/decorating it when it is decorated/undecorated gives $\mathsf{dr}(\varphi(P))=\mathsf{dr}(P)\mp 1$, but $\area(\varphi(P))=\area(P)\pm 1$. Since the decorations of the rises do not affect the dinv, we deduce
that $wt(\varphi(P))=-wt(P)$. This shows that in the sum $\sum_{P \in X}wt(P)$ all the contributions of the $P$ that are not fixed by $\varphi$ cancel out, leaving the sum over the fixed points of $\varphi$, i.e. over the paths with no rises.

The same argument applies to the sum $\sum_{P \in Y}wt(P)$.

This discussion proves the following theorem, which is the combinatorial counterpart of Theorem~\ref{thm:SFinvoDelta} and Theorem~\ref{thm:SFinvoDeltaSQ} under the Delta conjectures.
\begin{theorem}
	Given $m,n\in \mathbb{N}$, $m\geq 0$ and $n\geq 1$, we have
	\begin{equation}
	\sum_{s=0}^{n-1}(-t)^s\mathsf{PLSQ^E}_{\underline{x},q,t}(m,n)^{\ast s} = \mathsf{PLD}_{\underline{x},q,0}(m,n)^{\ast 0}
	\end{equation}
	and
	\begin{equation}
	\sum_{s=0}^{n-1}(-t)^s\mathsf{PLD}_{\underline{x},q,t}(m,n)^{\ast s} = \mathsf{PLD}_{\underline{x},q,0}(m,n)^{\ast 0}.
	\end{equation}
\end{theorem}
Combining this theorem with Theorem~\ref{thm:SFinvoDelta} and Theorem~\ref{thm:SFinvoDeltaSQ}, and with Proposition~\ref{prop:t0}, we get immediately the following curious corollary.
\begin{corollary}
	For fixed $m,n\in \mathbb{N}$, with $m\geq 0$ and $n> 0$, the truth of the generalized Delta (square) conjectures for all values of $k$ in $\{0,1,\dots,n-1\}$ except one imply the truth of the missing case.
\end{corollary}

\section{Open problems}

In force of the strong connection with the Delta conjecture, it is natural to ask for an analogue of any result or open problem about the Delta conjecture. We will not list all the possibilities here, but we refer to the literature on the Delta conjecture (e.g. the articles mentioned in the present work) for taking inspiration. 

Here we limit ourselves to mention the following problem: in the present work we proved the case $q=0$ of the Delta square conjecture (which reduced to the same case of the Delta conjecture); but notice that in general $[n-k]_t/[n]_t\Delta_{e_{n-k}}\omega(p_n)$ is not symmetric in $q$ and $t$, so, oddly enough, our work leaves open the case $t=0$ of the Delta square conjecture.

\section*{Appendix: a new proof of the Delta at $q=0$}

In this section we sketch a new proof of the Delta conjecture at $q=0$, i.e. of
\begin{equation} \label{eq:Deltaq0}
\left.\Delta_{e_{n-k-1}}'e_n\right|_{q=0}  = \mathsf{PLD}_{\underline x,0,t}(0,n)^{\ast k}.
\end{equation}

Our proof is different from both the original proof in \cite{Garsia-Haglund-Remmel-Yoo-2017} and the alternative one given in \cite{Haglund-Rhoades-Shimozono-arxiv}, though the general strategy is borrowed from the latter.

\medskip

\textbf{The strategy.} On the combinatorial side, in \cite{Haglund-Rhoades-Shimozono-Advances}*{Lemma~3.7} the authors proved essentially the following proposition, though using a different combinatorial interpretation.
\begin{proposition} \label{prop:deltaq0_Comb}
	For $j\geq 1$
	\begin{equation} \label{eq:comb_reco_Deltaq0}
	h_{j}^\perp \mathsf{PLD}_{\underline x,0,t}(0,n)^{\ast k}=\sum_{r=0}^jt^{\binom{j-r}{2}}\qbinom{n-k}{r}_t\qbinom{n-k-r}{j-r}_t \mathsf{PLD}_{\underline x,0,t}(0,n-j)^{\ast k-j+r},
	\end{equation}
	where $h_{j}^\perp$ is the adjoint operator with respect to the Hall scalar product of the multiplication by $h_j$.
\end{proposition}

In the next subsection we give a sketch of the proof using directly our definitions. 

Observe now that the identity \eqref{eq:Deltaq0} would follow easily from \eqref{eq:comb_reco_Deltaq0} and the following proposition, that we are going to prove later in this appendix.
\begin{proposition} \label{prop:deltaq0_SF}
	For $j\geq 1$
	\begin{equation} \label{eq:SF_reco_Deltaq0}
	h_{j}^\perp \left.\Delta_{e_{n-k-1}}'e_n\right|_{q=0}=\sum_{r=0}^jt^{\binom{j-r}{2}}\qbinom{n-k}{r}_t\qbinom{n-k-r}{j-r}_t \left.\Delta_{e_{n-k-r-1}}'e_{n-j}\right|_{q=0}.
	\end{equation}
\end{proposition}
\begin{remark} \label{rmk:qt_symmetry}
	Notice that the relation \eqref{eq:SF_reco_Deltaq0} is equivalent to the one where we exchange $q$ and $t$ everywhere: this is due to the well-known fact that the symmetric functions $\Delta_{e_{n-k-1}}'e_n$ are symmetric in $q$ and $t$ for all $n$ and $k$.
\end{remark}

Indeed, if two symmetric functions $f,g\in \Lambda^{(n)}$ with $n>0$ are such that $h_j^\perp f =h_j^\perp g$ for all $j\geq 1$, then it is not hard to see that we must have $f=g$ (cf. \cite{Haglund-Rhoades-Shimozono-Advances}*{Lemma~3.6}). Therefore, by induction on $n$, from Proposition~\ref{prop:deltaq0_Comb} and Proposition~\ref{prop:deltaq0_SF} we would deduce \eqref{eq:Deltaq0}.

\begin{remark}
	Observe that this is the same general strategy used in \cite{Haglund-Rhoades-Shimozono-arxiv}*{Theorem~4.2} to give an alternative proof of the Delta conjecture at $q=0$, though the authors use a relation similar but different from \eqref{eq:SF_reco_Deltaq0}. In any event, it should be noticed that our derivation of \eqref{eq:SF_reco_Deltaq0} will be completely different from what has been done in \cite{Haglund-Rhoades-Shimozono-arxiv} or in \cite{Garsia-Haglund-Remmel-Yoo-2017}.
\end{remark}

\begin{remark}
	Notice that the same argument, together with Remark~\ref{rmk:qt_symmetry} and Remark~\ref{rmk:comb_qt_symm}, proves also the Delta conjecture at $t=0$.
\end{remark}

So, in order to complete our proof of \eqref{eq:Deltaq0} we are going to prove Proposition~\ref{prop:deltaq0_Comb} and Proposition~\ref{prop:deltaq0_SF}: this is the content of the next two subsections.

\medskip

\textbf{Proof of Proposition~\ref{prop:deltaq0_Comb}.} We want to prove that for $j\geq 1$
\begin{equation}
h_{j}^\perp \mathsf{PLD}_{\underline x,0,t}(0,n)^{\ast k}=\sum_{r=0}^jt^{\binom{j-r}{2}}\qbinom{n-k}{r}_t\qbinom{n-k-r}{j-r}_t \mathsf{PLD}_{\underline x,0,t}(0,n-j)^{\ast k-j+r}.
\end{equation}

First of all, notice that, from general facts about superization \cite{Haglund-Book-2008}*{Chapter~6}, acting with $h_{j}^\perp$ on $\mathsf{PLD}_{\underline x,0,t}(0,n)^{\ast k}$ corresponds combinatorially to pick the elements $D$ in $\mathsf{PLD}(0,n)^{\ast k}$ in which the $j$ biggest labels appear in decreasing order in the dinv reading word (see Remark~\ref{rmk:dinv_reading} for the definition), and evaluating at $1$ the corresponding variables in $x^D$ (cf. the proof of \cite{Haglund-Rhoades-Shimozono-arxiv}*{Lemma~3.1}).

Let $D\in \mathsf{PLD}(0,n)^{\ast k}$ (i.e. $D$ is a partially labelled Dyck path of size $n$ with no zero valleys, and $k$ decorated rises) with dinv $0$, such that its dinv reading word is a shuffle of any permutation $\sigma \in \mathfrak{S}_{n-j}$ and a decreasing sequence $n, \dots, n-j+1$. Let us call \emph{big car} any label that is strictly greater than $n-j$, and \emph{small cars} the others.

Notice that all the big cars are necessarily peaks, hence they must lie in different columns. Also notice that the big cars are decreasing going bottom to top.

We need some definitions.

\begin{definition}
	Given a $D\in \mathsf{PLD}(0,n)^{\ast k}$ with dinv $0$, let $a=a(D) = a_1, \dots, a_n$ be its area word and $l = (l_1, \dots, l_n)$ be the sequence of its labels, where $l_i$ lies in the $i$-th row. Observe that such an object is characterized by the fact that the area word is weakly increasing, i.e. $a_i\leq a_{i+1}$ for $i=1,2,\dots,n-1$, and the inequalities $l_i>l_{i+1}$ when $a_i=a_{i+1}$.
	
	We say that an index $1 \leq i \leq n-1$ is \emph{contractible} if $a_{i-1} < a_{i} = a_{i+1}$ and $\ell_{i-1} < \ell_{i+1}$.
\end{definition}

Notice that contractible indices always correspond to peaks, thus we will refer to them as \emph{contractible peaks}.

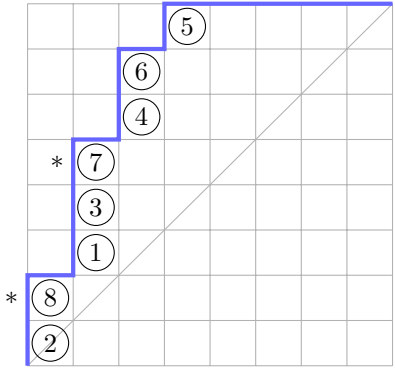
\begin{figure}[!ht]
	\centering
	
	\begin{tikzpicture}[scale=0.6]
	
	\draw[gray!60, thin] (0,0) grid (8,8);
	
	\draw[gray!60, thin] (0,0) -- (8,8);
	
	\draw[blue!60, line width = 1.6pt] (0,0) -- (0,1) -- (0,2) -- (1,2) -- (1,3) -- (1,4) -- (1,5) -- (2,5) -- (2,6) -- (2,7) -- (3,7) -- (3,8) -- (4,8) -- (5,8) -- (6,8) -- (7,8) -- (8,8);
	
	\draw (0.5,0.5) circle (0.4 cm) node {$2$};
	\draw (0.5,1.5) circle (0.4 cm) node {$8$};
	\draw (1.5,2.5) circle (0.4 cm) node {$1$};
	\draw (1.5,3.5) circle (0.4 cm) node {$3$};
	\draw (1.5,4.5) circle (0.4 cm) node {$7$};
	\draw (2.5,5.5) circle (0.4 cm) node {$4$};
	\draw (2.5,6.5) circle (0.4 cm) node {$6$};
	\draw (3.5,7.5) circle (0.4 cm) node {$5$};
	
	\node[left] at (0,1.5) {$\ast$};
	\node[left] at (1,4.5) {$\ast$};
	
	\end{tikzpicture}
	
	\caption{A parking function of size $8$ with dinv $0$ and $2$ decorated rises. The peaks labelled $6$ and $7$ are contractible. The peaks labelled $5$ and $8$ are not.}
	\label{fig:pf}
\end{figure}

We will now define a removing operation on the peaks of a labelled Dyck path with dinv $0$ as follows.

First of all, we choose a peak. Then we move all the decorations on the rises that lie weakly below that peak down by one rise: if the bottom-most rise is not decorated, the total number of decorated rises is preserved, and we call this \emph{rise-preserving removal}; otherwise we remove that decoration, letting the total number of decorated rises decrease by one, and we call this \emph{rise-killing removal}. If the peak is contractible, we remove the corresponding vertical step and the horizontal step immediately after it; otherwise, we remove the corresponding vertical step and the last horizontal step of the path. It is easy to see that the result of this procedure is still a labelled Dyck path with dinv $0$. See Figure~\ref{fig:pf} and Figure~\ref{fig:recursive_step} for an example of rise-killing removal. 
\begin{remark} \label{rmk:reversible}
	Notice that these two removal operations correspond, under the bijection defined in the proof of \cite{Haglund-Remmel-Wilson-2015}*{Proposition~4.1}, to the inverse of the insertions defined for the maj statistic in \cite{Wilson-Equidistribution}*{Section~4.2}. In any event, it is easy to see that these moves are indeed invertible once the loss of area is known.
\end{remark}
We describe now the \emph{removal algorithm} for the $j$ big cars. We apply the removing operation on the $j$ big cars (which are all necessarily peaks) on the bottom-most contractible peak among them, if any, and we repeat the procedure until there are no more contractible peaks. Then we apply the removing operation on the top-most non-contractible peak among them, if any, and we repeat the procedure until there are no more big cars.

\begin{figure}[!ht]
	\centering
	\begin{minipage}{0.36 \textwidth}
		\centering
		\begin{tikzpicture}[scale=0.55]
		\draw[gray!60, thin] (0,0) grid (8,8);
		
		\draw[gray!60, thin] (0,0) -- (8,8);
		
		\draw[blue!60, line width = 1.6pt] (0,0) -- (0,1) -- (0,2) -- (1,2) -- (1,3) -- (1,4) -- (1,5) -- (2,5) -- (2,6) -- (2,7) -- (3,7) -- (3,8) -- (4,8) -- (5,8) -- (6,8) -- (7,8) -- (8,8);
		
		\draw (0.5,0.5) circle (0.4 cm) node {$2$};
		\draw (0.5,1.5) circle (0.4 cm) node {$8$};
		\draw (1.5,2.5) circle (0.4 cm) node {$1$};
		\draw (1.5,3.5) circle (0.4 cm) node {$3$};
		\draw (1.5,4.5) circle (0.4 cm) node {$7$};
		\draw (2.5,5.5) circle (0.4 cm) node {$4$};
		\draw (2.5,6.5) circle (0.4 cm) node {$6$};
		\draw (3.5,7.5) circle (0.4 cm) node {$5$};
		
		\node[left] at (1,3.5) {$\ast$};
		
		\draw[->] (0.8,4.5) to [out=180, in=180] (0.4,3.5);
		\end{tikzpicture}
	\end{minipage}%
	\begin{minipage}{0.315 \textwidth}
		\centering
		\begin{tikzpicture}[scale=0.55]
		\draw[gray!60, thin] (0,0) grid (7,7);
		
		\draw[gray!60, thin] (0,0) -- (7,7);
		
		\draw[blue!60, line width = 1.6pt] (0,0) -- (0,1) -- (0,2) -- (1,2) -- (1,3) -- (1,4) -- (1,5) -- (1,6) -- (2,6) -- (2,7) -- (7,7);
		
		\draw (0.5,0.5) circle (0.4 cm) node {$2$};
		\draw (0.5,1.5) circle (0.4 cm) node {$8$};
		\draw (1.5,2.5) circle (0.4 cm) node {$1$};
		\draw (1.5,3.5) circle (0.4 cm) node {$3$};
		\draw (1.5,4.5) circle (0.4 cm) node {$4$};
		\draw (1.5,5.5) circle (0.4 cm) node {$6$};
		\draw (2.5,6.5) circle (0.4 cm) node {$5$};
		
		\node[left] at (1,3.5) {$\ast$};
		\end{tikzpicture}
	\end{minipage}%
	\begin{minipage}{0.28 \textwidth}
		\centering
		\begin{tikzpicture}[scale=0.55]
		\draw[gray!60, thin] (0,0) grid (6,6);
		
		\draw[gray!60, thin] (0,0) -- (6,6);
		
		\draw[blue!60, line width = 1.6pt] (0,0) -- (0,1) -- (1,1) -- (1,2) -- (1,3) -- (1,4) -- (1,5) -- (2,5) -- (2,6) -- (6,6);
		
		\draw (0.5,0.5) circle (0.4 cm) node {$2$};
		\draw (1.5,1.5) circle (0.4 cm) node {$1$};
		\draw (1.5,2.5) circle (0.4 cm) node {$3$};
		\draw (1.5,3.5) circle (0.4 cm) node {$4$};
		\draw (1.5,4.5) circle (0.4 cm) node {$6$};
		\draw (2.5,5.5) circle (0.4 cm) node {$5$};
		
		\node[left] at (1,3.5) {$\ast$};
		\end{tikzpicture}
	\end{minipage}
	
	\caption{The recursive step applied to the parking function in Figure~\ref{fig:pf} for $j=2$. It consists of a rise-killing removal on the peak with label $7$, followed by a rise-preserving removal on the peak with label $8$.}
	\label{fig:recursive_step}
\end{figure}
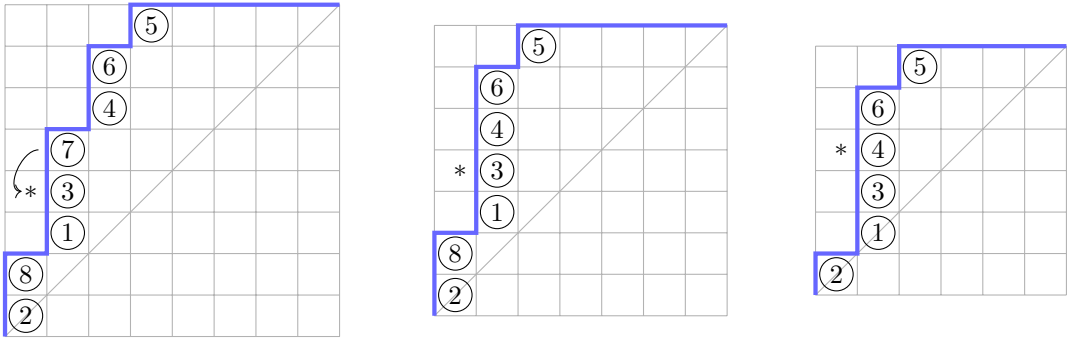

We claim that this algorithm is well defined, i.e. after we are done removing contractible peaks no big car can possibly become a contractible peak. In fact, after we are done removing contractible peaks no two big cars can be next to each other (otherwise the bottom-most one would be contractible), and the condition of being contractible only depends on the adjacent labels. Also notice that removing a contractible peak $i$ cannot create any contractible peak in any of the first $i-1$ rows.

We now want to compute the loss of the area given by our removal algorithm. We start by looking at what happens with a single removal of a big car.

When we remove a contractible peak $i$, the area first increases by the number of decorated rises that lie in the first $i$ rows (because all of them are moving to a rise at height exactly one less, except the bottom-most one, if decorated, which was at height one and disappears), then it decreases by $a_i$ (since it can't be a decorated rise any more). Notice that $a_i$ is equal to the number of rises in the first $i$ rows. It follows that the net area loss is given by the number of non-decorated rises in the first $i$ rows.

When we remove a non-contractible peak $i$ the area decreases by the same amount, plus the number of non-decorated vertical steps in rows $i+1$ to $n$ (since the corresponding letters of the area word are decreasing by one).

Let $r$ be the number of big cars that gets removed by our algorithm with a rise-preserving operation. We want to prove that the contributions of the $j-r$ rise-killing removals form a strictly increasing sequence of integers between $0$ and $n-k-r-1$, while the contributions of the $r$ rise-preserving removals form a weakly increasing sequence of integers between $0$ and $n-k-r$.

If we remove a contractible peak $i$, the number of non-decorated rises in the first $i$ rows can not possibly decrease (after removing the peak, the $i$-th row will contain a non-decorated rise by definition of contractible), and it increases by at least one if we perform a rise-killing removal (since the number of rises in the first $i$ rows is the same and the number of decorations is now one less). Furthermore we know that none of the first $i-1$ rows will contain a contractible peak, hence the new bottom-most contractible peak will have an index greater than or equal to $i$, which means that its contribution to the area is at least the same as the one of the peak we just removed, and it is in fact strictly greater if we performed a rise-killing removal.

Moreover, since we are removing non-contractible peaks top to bottom, the non-decorated vertical steps above each of them are weakly increasing, and the non-decorated rises that contributed for any of them still contribute (possibly as non-decorated steps strictly above instead of non-decorated rises weakly below), with the only possible exception of the last deleted peak; however, since big cars lie in different columns, we always have a valley (which is a non-decorated vertical step) between any two of them, and hence the contributions are weakly increasing. The same argument applies when we switch from contractible peaks to non-contractible ones. For the same reason as before, the contributions must strictly increase if we last performed a rise-killing removal, since the number of decorations decreases. 

The contribution of the last peak is at most the number of non-decorated vertical steps that are left (including itself) minus one (the first step is never counted, since it is weakly below the peak and it is not a rise), which is exactly $n-k-r$. If the last removal is rise-killing, however, it must be at least one unit smaller (the bottom-most rise is decorated, thus it doesn't contribute). 

It is easy to check that all the sequences can be achieved in this way, as the process is reversible once the losses of area are known (cf. Remark~\ref{rmk:reversible}).

Strictly increasing sequences of length $j-r$ of integers between $0$ and $n-k-r-1$ are $q$-counted by $q^{\binom{j-r}{2}} \qbinom{n-k-r}{j-r}_q$, while weakly increasing sequences of length $r$ of integers between $0$ and $n-k-r$ are $q$-counted by $\qbinom{n-k}{r}_q$. 

This completes the proof of Proposition~\ref{prop:deltaq0_Comb}.

\begin{remark} \label{rmk:comb_qt_symm}
	Notice that the relation \eqref{eq:comb_reco_Deltaq0} is true also when the roles of $q$ and $t$ are interchanged, i.e. 
	\begin{equation}
	h_{j}^\perp \mathsf{PLD}_{\underline x,q,0}(0,n)^{\ast k}=\sum_{r=0}^jq^{\binom{j-r}{2}}\qbinom{n-k}{r}_q\qbinom{n-k-r}{j-r}_q \mathsf{PLD}_{\underline x,q,0}(0,n-j)^{\ast k-j+r}.
	\end{equation}
	The argument is in fact slightly easier in this case, so we limit ourself to indicate the roles of the terms in the formula, leaving the details to the interested reader.
	
	The elements of $\mathsf{PLD}_{\underline x,q,0}(0,n)^{\ast k}$, i.e. the labelled Dyck paths of area $0$, are the ones for which the area word is a sequence of strictly increasing sequences all starting from $0$, and where all the rises are decorated, so that $n-k$ is the number of labels on the main diagonal. In the formula $r$ is the number of big cars on the diagonal, $\qbinom{n-k}{r}_q$ counts the dinv between the big cars and the small cars that lie on the main diagonal, $q^{\binom{j-r}{2}}\qbinom{n-k-r}{j-r}_q$ counts the dinv between the big cars and the small cars that are not on the diagonal, and $\mathsf{PLD}_{\underline x,q,0}(0,n-j)^{\ast k-j+r}$ keeps track of the remaining dinv among the small cars and the variables $\underline x$.
\end{remark}

\medskip

\textbf{Proof of Proposition~\ref{prop:deltaq0_SF}.} We want to prove that for $j\geq 1$
\begin{equation}
h_{j}^\perp \left.\Delta_{e_{n-k-1}}'e_n\right|_{q=0}=\sum_{r=0}^jt^{\binom{j-r}{2}}\qbinom{n-k}{r}_t\qbinom{n-k-r}{j-r}_t \left.\Delta_{e_{n-k-r-1}}'e_{n-j}\right|_{q=0}.
\end{equation}

\medskip

Using the expansion \eqref{eq:en_expansion}, it is easy to see (compare \cite{Haglund-Rhoades-Shimozono-arxiv}*{Equation~4.3}) that
\begin{align} \label{eq:specDeltaq0}
	\notag \left. \Delta_{e_{n-k-1}}'e_n\right|_{q=0}& =\sum_{\mu\vdash n} e_{n-k-1}[B_\mu-1] \left.\frac{M\Pi_\mu B_\mu\widetilde{H}_\mu[X;q,t]}{w_\mu}\right|_{q=0}\\
	&  = \sum_{\mu\vdash n} t^{\binom{n-k}{2}} \qbinom{\ell(\mu)-1}{n-k-1}_t\qbinom{\ell(\mu)}{m(\mu)}_t\widetilde{H}_\mu[X;0,t](-1)^{n-\ell(\mu)}t^{g(\mu)},
\end{align}
where we used \eqref{eq:espec} and the notation
\begin{equation}
\qbinom{\ell(\mu)}{m(\mu)}_t:=\qbinom{\ell(\mu)}{m_1(\mu),m_2(\mu),\dots,m_n(\mu)}_t,
\end{equation}
where $m_i(\mu)$ is defined to be number of parts of $\mu$ equal to $i$,
\begin{equation}
g(\mu):=-2n(\mu)-n+\sum_i\binom{m_i(\mu)+1}{2}\quad \text{ and }\quad n(\mu):=\sum_i \mu_i(i-1).
\end{equation}

We need the following lemma, that we are going to prove at the end of this Appendix.
\begin{lemma} \label{lem:tech}
	Given $\nu\vdash n$ and $j\geq 1$, we have
	\begin{align} \label{eq:lemmino}
		\left.\sum_{\mu \supset_j \nu}e_{n-k-1}[B_\mu-1]B_\mu\Pi_\mu d_{\mu \nu}^{(j)}\right|_{q=0} & = \Pi_\nu(0,t)\cdot t^{\binom{n-k-j}{2}} \qbinom{\ell(\nu)+j-1}{n-k-1}_t \qbinom{n-k}{j}_t[\ell(\nu)]_t .
	\end{align}
\end{lemma}
We now have
\begin{align*}
	h_j^\perp \left.\Delta_{e_{n-k-1}}'e_n\right|_{q=0} & = \sum_{\mu\vdash n}e_{n-k-1}[B_\mu-1]\frac{MB_\mu\Pi_\mu }{w_\mu} h_j^\perp \left.\widetilde{H}_\mu[X;q,t] \right|_{q=0}\\
	\text{(using \eqref{eq:def_cmunu})}& = \sum_{\mu\vdash n}e_{n-k-1}[B_\mu-1]MB_\mu\Pi_\mu  \sum_{\nu\subset_j \mu}\frac{c_{\mu \nu}^{(j)}}{w_\mu} \left.\widetilde{H}_\nu[X;q,t] \right|_{q=0}\\
	\text{(using \eqref{eq:rel_cmunu_dmunu})}& = \sum_{\nu\vdash n-j} \frac{\widetilde{H}_\nu[X;0,t]}{w_\nu(0,t)}(1-t)\sum_{\mu\supset_j \nu} \left.e_{n-k-1}[B_\mu-1]B_\mu\Pi_\mu  d_{\mu \nu}^{(j)} \right|_{q=0}  \\
	\text{(using \eqref{eq:lemmino})}& = \sum_{\nu\vdash n-j} \frac{\widetilde{H}_\nu[X;0,t]}{w_\nu(0,t)}(1-t)\Pi_\nu(0,t)\cdot t^{\binom{n-k-j}{2}} \qbinom{\ell(\nu)+j-1}{n-k-1}_t \qbinom{n-k}{j}_t[\ell(\nu)]_t \\
	& = t^{\binom{j}{2}}\qbinom{n-k}{j}_t \sum_{\nu\vdash n-j} t^{\binom{n-k}{2}-j(n-k-1)}
	\qbinom{\ell(\nu)+j-1}{n-k-1}_t \times\\
	& \quad \times  \qbinom{\ell(\nu)}{m(\nu)}_t\widetilde{H}_\nu[X;0,t]  (-1)^{n-j-\ell(\nu)}t^{g(\nu)},
\end{align*}
where in the last equality we used
\begin{align}
	\binom{n-k-j}{2} 
	&=  \binom{n-k}{2} +\binom{j}{2} - j(n-k-1).
\end{align}

Recalling \eqref{eq:specDeltaq0}, in order to get \eqref{eq:SF_reco_Deltaq0}, we must have
\begin{align*}
	& \hspace{-2cm}t^{\binom{n-k}{2}-(n-k-1)j}\qbinom{\ell(\nu)+j-1}{n-k-1}_t t^{\binom{j}{2}}\qbinom{n-k}{j}_t=\\
	& = \sum_{r=0}^jt^{\binom{j-r}{2}}\qbinom{n-k}{r}_t\qbinom{n-k-r}{j-r}_t t^{\binom{n-k-r}{2}}\qbinom{\ell(\nu)-1}{n-k-r-1}_t.
\end{align*}
But using the substitution $r=n-k-m$ we have
\begin{align*}	
	& \hspace{-2cm}	\sum_{r=0}^jt^{\binom{j-r}{2}}\qbinom{n-k}{r}_t\qbinom{n-k-r}{j-r}_t t^{\binom{n-k-r}{2}}\qbinom{\ell(\nu)-1}{n-k-r-1}_t=\\
	& = \sum_{m=n-k-j}^{n-k}t^{\binom{j-n+k+m}{2}}\qbinom{n-k}{n-k-m}_t\qbinom{m}{j-n+k+m}_t t^{\binom{m}{2}}\qbinom{\ell(\nu)-1}{m-1}_t\\
	& = t^{\binom{n-k}{2}-j(n-k-1)+\binom{j}{2}}\qbinom{n-k}{j}_t \times\\
	& \quad \times \sum_{m \geq 1} t^{(m-n+k+j)(m-1)} \qbinom{j}{n-k-m}_t\qbinom{\ell(\nu)-1}{m-1}_t
\end{align*}
where in the last equality we used an easy manipulation of $t$-binomials and 
\begin{align*}
	\binom{j-n+k+m}{2} +\binom{m}{2}
	& = \binom{n-k}{2}-j(n-k-1)+\binom{j}{2}+(m-n+k+j)(m-1).
\end{align*}
So we are left to show
\begin{align}
	\qbinom{\ell(\nu)+j-1}{n-k-1}_t & = \sum_{m \geq 1} t^{(m-n+k+j)(m-1)} \qbinom{j}{n-k-m}_t\qbinom{\ell(\nu)-1}{m-1}_t,
\end{align}
which is none other than the well-known $q$-Vandermonde (cf. \cite{Andrews-Book-Partitions}*{Equation~(3.3.10)}).

\medskip

To complete our proof of the Delta conjecture at $q=0$ it remains only to prove Lemma~\ref{lem:tech}.

\medskip

\textbf{Proof of Lemma~\ref{lem:tech}.} Given $\nu\vdash n$ and $j\geq 1$, we want to prove
\begin{align} 
	\hspace{-0.5cm}\left.\sum_{\mu \supset_j \nu}e_{n-k-1}[B_\mu-1]B_\mu\Pi_\mu d_{\mu \nu}^{(j)}\right|_{q=0} & = \Pi_\nu(0,t)[\ell(\nu)]_t \cdot t^{\binom{n-k-j}{2}} \qbinom{\ell(\nu)+j-1}{n-k-1}_t \qbinom{n-k}{j}_t\\
	& = \Pi_\nu(0,t)[n-k]_t \cdot t^{\binom{n-k-j}{2}} \qbinom{\ell(\nu)+j-1}{j}_t \qbinom{\ell(\nu)}{n-k-j}_t.
\end{align}
We will need two more identities: \cite{DAdderio-VandenWyngaerd-2017}*{Lemma~5.2}, i.e.
\begin{equation} \label{eq:lem52}
e_{n-k-1}[B_\beta-1]B_\beta=\sum_{\gamma\subset_k \beta}c_{\beta \gamma}^{(k)}B_\gamma T_\gamma\quad \text{ for } \beta\vdash n>k\geq 1 ,
\end{equation}

and \cite{Haglund-Schroeder-2004}*{Theorem~2.6}, i.e. for any $A,F\in \Lambda$ homogeneous
\begin{equation} \label{eq:HaglundThm}
\sum_{\mu\vdash n}\Pi_\mu F[MB_\mu]d_{\mu\nu}^A=\Pi_\nu\left(\Delta_{A[MX]}F[X]\right)[MB_\nu],
\end{equation}
where $d_{\mu\nu}^A$ is the generalized Pieri coefficient defined by
\begin{equation}
\sum_{\mu\supset \nu}d_{\mu\nu}^A\widetilde{H}_\mu=A\widetilde{H}_\nu.
\end{equation}

Setting $A[X]=e_j[X/M]$ and $F[X]=e_{n-k-1}[X/M-1]e_1[X/M]$ in \eqref{eq:HaglundThm}, we get
\begin{align} \label{eq:summands}
	\notag	\sum_{\mu \supset_j \nu}e_{n-k-1}[B_\mu-1]B_\mu\Pi_\mu d_{\mu \nu}^{(j)} & = \Pi_\nu \left( \Delta_{e_j}e_{n-k-1}[X/M-1]e_1[X/M]\right)[MB_\nu]\\
	\notag	\text{(using \eqref{eq:deltaprime})}& = \Pi_\nu \left( \sum_{i=0}^{n-k-1}\Delta_{e_j}(-1)^{n-k-1-i}e_i[X/M]e_1[X/M]\right)[MB_\nu]\\
	\notag	\text{(using \eqref{eq:e_h_expansion})}& = \Pi_\nu \left( \sum_{i=0}^{n-k-1}\Delta_{e_j}(-1)^{n-k-1-i}\sum_{\beta\vdash i}e_1[X/M]\frac{\widetilde{H}_\beta[X]}{w_\beta}\right)[MB_\nu]\\
	\notag	\text{(using \eqref{eq:def_dmunu})} & = \Pi_\nu \left( \sum_{i=0}^{n-k-1}\Delta_{e_j}(-1)^{n-k-1-i}\sum_{\beta\vdash i}\sum_{\gamma\supset_1 \beta}\frac{d_{\gamma \beta}^{(1)}}{w_\beta}\widetilde{H}_\gamma[X] \right)[MB_\nu]\\
	\notag	\text{(using \eqref{eq:rel_cmunu_dmunu})} & = \Pi_\nu \left( \sum_{i=0}^{n-k-1}(-1)^{n-k-1-i}\sum_{\gamma\vdash i+1}\sum_{\beta\subset_1 \gamma}\frac{c_{\gamma \beta}^{(1)}}{w_\gamma}e_j[B_\gamma]\widetilde{H}_\gamma[X] \right)[MB_\nu]\\
	\notag \text{(using \eqref{eq:sumBmu})}	& = \Pi_\nu \sum_{i=0}^{n-k-1}(-1)^{n-k-1-i}\sum_{\gamma\vdash i+1}B_\gamma e_j[B_\gamma]\frac{\widetilde{H}_\gamma[MB_\nu]}{w_\gamma} \\
	\text{(using \eqref{eq:deltaprime})}& = \Pi_\nu \sum_{i=0}^{n-k-1}(-1)^{n-k-1-i}\sum_{\gamma\vdash i+1}B_\gamma e_j[B_\gamma-1]\frac{\widetilde{H}_\gamma[MB_\nu]}{w_\gamma}\\
	\notag	& \quad +  \Pi_\nu \sum_{i=0}^{n-k-1}(-1)^{n-k-1-i}\sum_{\gamma\vdash i+1}B_\gamma  e_{j-1}[B_\gamma-1] \frac{\widetilde{H}_\gamma[MB_\nu]}{w_\gamma}.
\end{align}

Now, taking the first of the two summands in \eqref{eq:summands}, we have:
\begin{align*}
	& \hspace{-2cm} \Pi_\nu \sum_{i=0}^{n-k-1}(-1)^{n-k-1-i}\sum_{\gamma\vdash i+1}B_\gamma e_j[B_\gamma-1]\frac{\widetilde{H}_\gamma[MB_\nu]}{w_\gamma}= \\
	\text{(using \eqref{eq:lem52})}& = \Pi_\nu \sum_{i=0}^{n-k-1}(-1)^{n-k-1-i}\sum_{\gamma\vdash i+1} \sum_{\alpha\subset_{i-j} \gamma}c_{\gamma \alpha}^{(i-j)}B_\alpha T_\alpha \frac{\widetilde{H}_\gamma[MB_\nu]}{w_\gamma} \\
	\text{(using \eqref{eq:rel_cmunu_dmunu})}& = \Pi_\nu \sum_{i=0}^{n-k-1}(-1)^{n-k-1-i}\sum_{\alpha\vdash j+1}\frac{B_\alpha T_\alpha}{w_\alpha}  \sum_{\gamma\supset_{i-j} \alpha}d_{\gamma \alpha}^{(i-j)}\widetilde{H}_\gamma[MB_\nu] \\
	\text{(using \eqref{eq:def_dmunu})} & = \Pi_\nu \sum_{i=0}^{n-k-1}(-1)^{n-k-1-i}e_{i-j}[B_\nu] \sum_{\alpha \vdash j+1}T_\alpha B_\alpha \frac{\widetilde{H}_\alpha[MB_\nu]}{w_\alpha} \\
	\text{(using \eqref{eq:deltaprime})}& = \Pi_\nu e_{n-k-1-j}[B_\nu-1] \sum_{\alpha \vdash j+1}T_\alpha B_\alpha \frac{\widetilde{H}_\alpha[MB_\nu]}{w_\alpha}. 
\end{align*}

When we specialize this at $q=0$, because of the obvious 
\begin{equation}
T_\alpha(0,t)=\delta_{\alpha,(1^{j+1})}t^{\binom{j+1}{2}},
\end{equation} 
the only term that survives in the sum is the one with $\alpha=(1^{j+1})$. Now using \eqref{eq:Hn}, the well-known
\begin{equation}
\widetilde{H}_{\mu}[X;q,t]=\widetilde{H}_{\mu'}[X;t,q],
\end{equation}
and the obvious $w_{(1^{j+1})}=\prod_{i=1}^{{j+1}}(1-t^i) \cdot \prod_{i=0}^{j}(t^i-q)$, we get
\begin{align*}
	&\hspace{-2cm} \Pi_\nu  e_{n-k-1-j}[B_\nu-1]  \sum_{\alpha \vdash j+1}T_\alpha B_\alpha \left.\frac{\widetilde{H}_\alpha[MB_\nu]}{w_\alpha}\right|_{q=0}= \\  
	& = \Pi_\nu(0,t) e_{n-k-1-j}[[\ell(\nu)]_t-1] t^{\binom{j+1}{2}} \frac{[j+1]_t}{w_{1^{j+1}}(0,t)} h_{j+1}[[\ell(\nu)]_t]\prod_{i=1}^{j+1}(1-t^i)\\
	& = \Pi_\nu(0,t)  e_{n-k-1-j}[[\ell(\nu)]_t-1] [j+1]_t h_{j+1}[[\ell(\nu)]_t] \\
	\text{(using \eqref{eq:espec} and \eqref{eq:h_q_binomial})}& = \Pi_\nu(0,t) t^{\binom{n-k-j}{2}}\qbinom{\ell(\nu)-1}{n-k-1-j}_t [j+1]_t \qbinom{\ell(\nu)+j}{j+1}_t .
\end{align*}
Of course for the second summand in \eqref{eq:summands} we get the same result with $j$ replaced by $j-1$, so that

\begin{align*}
	& \hspace{-1cm}\sum_{\mu \supset_j \nu}\left.e_{n-k-1}[B_\mu-1]B_\mu\Pi_\mu d_{\mu \nu}^{(j)}\right|_{q=0}=\\
	& =\Pi_\nu(0,t)   t^{\binom{n-k-j}{2}}\qbinom{\ell(\nu)-1}{n-k-1-j}_t [j+1]_t \qbinom{\ell(\nu)+j}{j+1}_t\\
	& \quad + \Pi_\nu(0,t)   t^{\binom{n-k-j+1}{2}}\qbinom{\ell(\nu)-1}{n-k-j}_t [j]_t \qbinom{\ell(\nu)+j-1}{j}_t\\
	& =\Pi_\nu(0,t)   t^{\binom{n-k-j}{2}}\qbinom{\ell(\nu)+j-1}{j}_t\left(\qbinom{\ell(\nu)-1}{n-k-1-j}_t [\ell(\nu)+j]_t+t^{n-j-k} \qbinom{\ell(\nu)-1}{n-k-j}_t [j]_t\right)
\end{align*}
but
\begin{align*}
	& \hspace{-2cm}	\qbinom{\ell(\nu)-1}{n-k-1-j}_t [\ell(\nu)+j]_t+t^{n-j-k} \qbinom{\ell(\nu)-1}{n-k-j}_t [j]_t= \\
	& =\qbinom{\ell(\nu)-1}{n-k-1-j}_t [\ell(\nu)+j]_t+ [j]_t\left(\qbinom{\ell(\nu)}{n-k-j}_t- \qbinom{\ell(\nu)-1}{n-k-j-1}_t\right)\\
	& =\qbinom{\ell(\nu)-1}{n-k-1-j}_t [\ell(\nu)]_tt^j+ [j]_t \qbinom{\ell(\nu)}{n-k-j}_t\\
	& =\qbinom{\ell(\nu)}{n-k-j}_t [n-k-j]_tt^j+ [j]_t \qbinom{\ell(\nu)}{n-k-j}_t\\
	& =[n-k]_t \qbinom{\ell(\nu)}{n-k-j}_t,
\end{align*}
where in the first equality we used the well-known $q^b\qbinom{a-1}{b}+\qbinom{a-1}{b-1}=\qbinom{a}{b}$. Hence 
\begin{align*}
	\sum_{\mu \supset_j \nu}\left.e_{n-k-1}[B_\mu-1]B_\mu\Pi_\mu d_{\mu \nu}^{(j)}\right|_{q=0} & =\Pi_\nu(0,t)   t^{\binom{n-k-j}{2}}\qbinom{\ell(\nu)+j-1}{j}_t\qbinom{\ell(\nu)}{n-k-j}_t[n-k]_t\\
	& =\Pi_\nu(0,t)   t^{\binom{n-k-j}{2}}\qbinom{\ell(\nu)+j-1}{n-k-1}_t\qbinom{n-k}{j}_t[\ell(\nu)]_t,
\end{align*}
where in the last equality we used an easy manipulation of $t$-binomials (cf \cite{DAdderio-VandenWyngaerd-2017}*{Lemma~4.3}). 

This completes the proof of Lemma~\ref{lem:tech}.


\bibliographystyle{amsalpha}
\bibliography{Biblebib}

\end{document}